\documentclass{amsart}

\newtheorem{theorem}{Theorem}[section]
\newtheorem{proposition}[theorem]{Proposition}
\newtheorem{corollary}[theorem]{Corollary}

\theoremstyle{definition}
\newtheorem{definition}[theorem]{Definition}

\theoremstyle{remark}
\newtheorem{remark}[theorem]{Remark}
\newtheorem{remarks}[theorem]{Remarks}

\numberwithin{equation}{section}

\DeclareMathOperator{\Spec}{Spec}

\begin{document}

\title{Canonical singularities of dimension three 
in \\ characteristic $2$ which  do not follow \\ Reid's rules}

\author{Masayuki Hirokado}
\address{Graduate school of Information Sciences, Hiroshima City University, Ozuka-higashi, Asaminami-ku, Hiroshima 731-3194, Japan}
\email{hirokado@math.info.hiroshima-cu.ac.jp}






\begin{abstract}
We continue to study and present concrete examples 
in characteristic $2$ of 
compound Du Val singularities 
defined over an algebraically closed field 
which have one dimensional singular loci 
but cannot be written as  
products (a rational double point)$\times$(a curve)
up to analytic isomorphism at any point of the loci.
Unlike in other characteristics,
we find a large number of such examples
whose general hyperplane sections have 
rational double points
of type $D$.
We consider these compound Du Val singularities 
as a special class of canonical singularities, 
intend to complete classification \cite[Theorem~3]{HIS13}
in arbitrary characteristic
reinforcing Miles Reid's result 
in characteristic zero \cite[(1.14)]{Reid80}.
\end{abstract}

\maketitle

\markboth{MASAYUKI HIROKADO}{CANONICAL SINGULARITIES IN CHARACTERISTIC $2$}


.

\section{Introduction}
We consider varieties defined over an 
algebraically closed field $k$  
of characteristic $p\geq 0$.
The hypersurface singularity in $\mathbf A_k^4$ with $p=3$ given respectively by 
\begin{eqnarray*}
&&z^2+x^3+y^4+wy^3=0,\quad\\ 
&&z^2+x^3+y^4+x^2y^2+wy^3=0
\end{eqnarray*}
has a one dimensional singular locus $C$, and posesses a crepant resolution.
Simple observation of blow-ups finds 
that these cannot be written as fiber products 
(a rational double point) $\times$ (a curve)
up to analytic isomorphism at any $x\in C$.
These are the examples of canonical singularities
in positive characteristic  
which show that Miles Reid's result in characteristic~$0$
needs some modification when one generalizes it to 
arbitrary characteristic.
  
\begin{theorem}[Reid, \cite{Reid80}]
Let $X$ be a quasi-projective normal variety of dimension three
defined over  the field of complex numbers $\mathbf C$.
If $X$ has at most canonical singularities, then 
there exists some zero dimensional variety $Z\subset X$ such that for any  $x\in X\setminus Z$, 
the complete local ring $\hat{\mathcal O}_{X, x}$ 
is either regular or isomorphic to $\mathbf C[[x,y,z]]/(g(x,y,z))\hat\otimes \mathbf C[[w]]$, 
where $\Spec \mathbf C[[x,y,z]]/(g(x,y,z))$ is a rational double point.
\end{theorem}
 
One way to study such singularities is to examine versal 
deformations of rational double points.
In the preceding work, we proved that these examples are also
exhaustive in $p \geq 3$, i.e. 
there is no such example in $p\geq 5$ 
and in $p=3$ there is no other in the following sense. 

\begin{theorem}[HIS, \cite{HIS13}]
Let $X$ be a three dimensional normal algebraic variety 
over an algebraically closed field $k$ of characteristic $p>0$
with at most canonical singularities.
Then the following assertions hold. 
\begin{enumerate}
\item[{\rm i)}] $X$ is Cohen-Macaulay in codimension two.
\item[{\rm ii)}] Suppose $p>2$ and that a general hyperplane section $H$
of $X$ has at most rational singularities.
Then there exists a zero-dimensional subvariety $Z\subset X$
such that for any $x\in X\setminus Z$, the complete
local ring $\hat{\mathcal O}_{X,x}$ is either regular or
isomorphic to one of the following:
\begin{itemize}
\item $k[[x,y,z]]/(g)\hat{\otimes}k[[w]]$ 
with a rational double point $\Spec k[[x,y,z]]/(g)$,
\item $k[[w, x,y,z]]/(z^2+x^3+y^4+x^2y^2+wy^3)$ with $p=3$,
\item $k[[w, x,y,z]]/(z^2+x^3+y^4+wy^3)$ with $p=3$.
\end{itemize}
\end{enumerate}   
\end{theorem}

\noindent
We call such deviations non-classical compound Du Val 
singularities.

\begin{table}[h]
\caption{Non-classical compound Du Val singularities in $p=3$. }
\label{tab:nonclassicalcdv01}
\begin{center}
{\footnotesize
\begin{tabular}{l|l|l}
\hline
   & \textrm Type & \textrm Defining equation\\
\hline
 $E_6^1$ & $\mathbf{E_6^1G_2}$  & $z^2+x^3+y^4+x^2y^2+wy^3=0$  \\
\cline{2-3}
 $E_6^0$ & $\mathbf{E_6^0G_2}$  & $z^2+x^3+y^4+wy^3=0$  \\
\hline
\end{tabular}
}
\end{center}
\end{table}

\noindent
We want to complete classification.
As a first step, we present in this paper concrete examples  in $p=2$. 

{\bf Theorem \ref{main}} is our main result.
See also tables 2, 3 and 4 where equations are summarized.

\begin{table}[h]
\caption{Non-classical compound Du Val singularities in $p=2$, II. }
\label{tab:nonclassicalcdv02}
\begin{center}
{\footnotesize
\begin{tabular}{l|l|l}
\hline
   & \textrm Type & \textrm Defining equation\\
\hline
 $E_7^3$ & $\mathbf{E_7^3F_4}$  & $z^2+x^3+xy^3+xyz+wy^4=0$  \\
\cline{2-3}
 $E_7^2$ & $\mathbf{E_7^2F_4}$  & $z^2+x^3+xy^3+y^3z+wy^4=0$  \\
\cline{2-3}
 $E_7^1$ & $\mathbf{E_7^1F_4}$  & $z^2+x^3+xy^3+wy^4=0$  \\
\hline
 $E_8^3$ & $\mathbf{E_8^3F_4}$  & $z^2+x^3+y^3z+wy^4=0$  \\
\cline{2-3}
 $E_8^2$ & $\mathbf{E_8^2F_4}$  & $z^2+x^3+xy^2z+wy^4=0$  \\
\cline{2-3}
 $E_8^1$ & $\mathbf{E_8^1F_4}$  & $z^2+x^3+xy^3z+wy^4=0$  \\
\cline{2-3}
 $E_8^0$ & $\mathbf{E_8^0F_4}$  & $z^2+x^3+wy^4=0$  \\
\hline
\end{tabular}
}
\end{center}
\end{table}

As one can see, there are quite a few compared with two equations in $p=3$. 
Among them six were already known \cite[Theorem 2]{HIS13},
which are marked by boxes in the main theorem,
but we believe others are new.  

As for the question of rationality, it might be suggestive that 
if one sticks to traditional notion, 
these examples exhibit no peculiarity.

\textbf{Corollary~\ref{vanishing}}
{\it Let $X$ be a hypersurface singularity given by one of the equations 
in Theorem~\ref{main}.
Then the following assertions hold.
\begin{itemize}
\item [i)] $R^i\pi_*\mathcal O_{\tilde X}=0$ $(i>0)$ holds for any resolution of singularities
$\pi:\tilde X\to X$. 
\item [ii)] $R^i\pi_*K_{\tilde X}=0$ $(i>0)$ holds for any resolution of singularities
$\pi:\tilde X\to X$. 
\end{itemize}
}

\noindent
However, if one asks about $F$-rationality 
(cf. \cite[\S 10]{BrunsHerzog98}), 
these examples turn out to be highly irrational.

\textbf{Proposition~\ref{f-rationality}}
{\it Let $X$ be a hypersurface singularity given by one of the equations 
in Theorem~\ref{main}.
Then the following assertions hold.
\begin{itemize}
\item [i)] $X$ is $F$-pure if and only if 
the type is one of 
{\footnotesize
$\mathbf{E_7^3F_4}$, $\mathbf{D_4^1B_1}$, 
$\mathbf{D_4^1B_2}$, $\mathbf{D_5^1C_3}$, $\mathbf{D_{2n}^{n-1}B_{n-1}}$,
$\mathbf{D_{2n}^{n-1}B_{n}}$, $\mathbf{D_{2n}^{n-1}C_{2n-2}}$,
$\mathbf{D_{2n+1}^{n-1}B_{n-1}}$,   
$\mathbf{D_{2n+1}^{n-1}C_{2n-1}}$ 
}
with $n\geq 3$.
\item [ii)] $X$ is not $F$-rational.
\end{itemize}
}

\section{Preliminaries}

Standard notation and terminologies found for example
in \cite{Hartshorne77}, \cite{Matsumura86}
should be used without mentioning in this paper.   
\begin{table}[h]
\caption{Non-classical compound Du Val singularities in $p=2$, I. }
\label{tab:nonclassicalcdv03}
\begin{center}
{\scriptsize
\begin{tabular}{l|l|l}
\hline
   & \textrm Type & \textrm Defining equation\\
\hline
 $D_4^1$ & $\mathbf{D_4^1B_1}$  & $z^2+xyz+wx^2+y^3=0$  \\
  & $\mathbf{D_4^1B_2}$  & $z^2+xy^2+xyz+wx^2=0$  \\
\cline{2-3}
 $D_4^0$ & $\mathbf{D_4^0B_1}$ &  $z^2+wx^2+y^3=0$  \\
       & $\mathbf{D_4^0B_2}$ &  $z^2+xy^2+wx^2=0$  \\
\hline
 $D_5^1$ & $\mathbf{D_5^1C_3}$ &  $z^2+x^2y+y^2z+xyz+wy^3=0$\\
\cline{2-3}
 $D_5^0$ & $\mathbf{D_5^0C_3}$ &  $z^2+x^2y+y^2z+wy^3=0$\\
\hline
 $D_6^2$ & $\mathbf{D_6^2B_2}$  &  $z^2+xy^3+xyz+wx^2+y^5=0$ \\
         & $\mathbf{D_6^2B_3}$  &  $z^2+xy^3+xyz+wx^2=0$ \\
         & $\mathbf{D_6^2C_4}$  &  $z^2+x^2y+xy^3+xyz+wy^4=0$ \\
\cline{2-3}
 $D_6^1$ & $\mathbf{D_6^1B_2}$  &  $z^2+xy^3+xy^2z+wx^2+y^5=0$ \\
         & $\mathbf{D_6^1B_3}$  &  $z^2+xy^3+xy^2z+wx^2=0$ \\
         & $\mathbf{D_6^1C_3}$  &  $z^2+x^2y+xy^3+xy^2z+wy^3=0$ \\
         & $\mathbf{D_6^1C_4}$  &  $z^2+x^2y+xy^3+xy^2z+wy^4=0$ \\
\cline{2-3}
$D_6^0$ & $\mathbf{D_6^0C_3}$  &  $z^2+x^2y+xy^3+wy^3=0$\\
        & $\mathbf{D_6^0C_4}$  &  $z^2+x^2y+xy^3+wy^4=0$\\
\hline
 $D_7^2$ & $\mathbf{D_7^2B_2}$ &  $z^2+xyz+wx^2+y^5=0$\\
         & $\mathbf{D_7^2C_5}$ &  $z^2+x^2y+xyz+y^3z+wy^5=0$\\
\cline{2-3}
 $D_7^1$ & $\mathbf{D_7^1C_4}$ &  $z^2+x^2y+xy^2z+y^3z+wy^4=0$\\
         & $\mathbf{D_7^1C_5}$ &  $z^2+x^2y+xy^2z+y^3z+wy^5=0$\\
\cline{2-3}
 $D_7^0$ & $\mathbf{D_7^0C_4}$ &  $z^2+x^2y+y^3z+wy^4=0$\\
  & $\mathbf{D_7^0C_5}$ &  $z^2+x^2y+y^3z+wy^5=0$\\
\hline
 $D_8^3$ & $\mathbf{D_8^3B_3}$ &  $z^2+xy^4+xyz+wx^2+y^7=0$\\
         & $\mathbf{D_8^3B_4}$ &  $z^2+xy^4+xyz+wx^2=0$\\
         & $\mathbf{D_8^3C_6}$ &  $z^2+x^2y+xy^4+xyz+wy^6=0$\\
\cline{2-3}
 $D_8^2$ & $\mathbf{D_8^2B_2}$ &  $z^2+xy^4+xy^2z+wx^2+y^5=0$\\
         & $\mathbf{D_8^2B_3}$ &  $z^2+xy^4+xy^2z+wx^2+y^7=0$\\
         & $\mathbf{D_8^2B_4}$ &  $z^2+xy^4+xy^2z+wx^2=0$\\
         & $\mathbf{D_8^2C_4}$ &  $z^2+x^2y+xy^4+xy^2z+wy^4=0$\\
         & $\mathbf{D_8^2C_5}$ &  $z^2+x^2y+xy^4+xy^2z+wy^5=0$\\
         & $\mathbf{D_8^2C_6}$ &  $z^2+x^2y+xy^4+xy^2z+wy^6=0$\\
\cline{2-3}
 $D_8^1$ & $\mathbf{D_8^1B_2}$ &  $z^2+xy^4+xy^3z+wx^2+y^5=0$\\
         & $\mathbf{D_8^1B_3}$ &  $z^2+xy^4+xy^3z+wx^2+y^7=0$\\
         & $\mathbf{D_8^1B_4}$ &  $z^2+xy^4+xy^3z+wx^2=0$\\
         & $\mathbf{D_8^1C_4}$ &  $z^2+x^2y+xy^4+xy^3z+wy^4=0$\\
         & $\mathbf{D_8^1C_5}$ &  $z^2+x^2y+xy^4+xy^3z+wy^5=0$\\
         & $\mathbf{D_8^1C_6}$ &  $z^2+x^2y+xy^4+xy^3z+wy^6=0$\\
\cline{2-3}
 $D_8^0$ & $\mathbf{D_8^0B_2}$ &  $z^2+xy^4+wx^2+y^5=0$\\
         & $\mathbf{D_8^0B_3}$ &  $z^2+xy^4+wx^2+y^7=0$\\
         & $\mathbf{D_8^0B_4}$ &  $z^2+xy^4+wx^2=0$\\
         & $\mathbf{D_8^0C_4}$ &  $z^2+x^2y+xy^4+wy^4=0$\\
         & $\mathbf{D_8^0C_5}$ &  $z^2+x^2y+xy^4+wy^5=0$\\
         & $\mathbf{D_8^0C_6}$ &  $z^2+x^2y+xy^4+wy^6=0$\\
 \hline
$D_9^3$ & $\mathbf{D_9^3B_3}$ &  $z^2+xyz+wx^2+y^7=0$\\
        & $\mathbf{D_9^3C_7}$ &  $z^2+x^2y+xyz+y^4z+wy^7=0$\\
\cline{2-3}
$D_9^2$ & $\mathbf{D_9^2C_5}$ &  $z^2+x^2y+xy^2z+y^4z+wy^5=0$\\
        & $\mathbf{D_9^2C_6}$ &  $z^2+x^2y+xy^2z+y^4z+wy^6=0$\\
        & $\mathbf{D_9^2C_7}$ &  $z^2+x^2y+xy^2z+y^4z+wy^7=0$\\
\cline{2-3}
$D_9^1$ & $\mathbf{D_9^1C_5}$ &  $z^2+x^2y+xy^3z+y^4z+wy^5=0$\\
        & $\mathbf{D_9^1C_6}$ &  $z^2+x^2y+xy^3z+y^4z+wy^6=0$\\
        & $\mathbf{D_9^1C_7}$ &  $z^2+x^2y+xy^3z+y^4z+wy^7=0$\\
\cline{2-3}
$D_9^0$ & $\mathbf{D_9^0C_5}$ &  $z^2+x^2y+y^4z+wy^5=0$\\
        & $\mathbf{D_9^0C_6}$ &  $z^2+x^2y+y^4z+wy^6=0$\\
        & $\mathbf{D_9^0C_7}$ &  $z^2+x^2y+y^4z+wy^7=0$\\
\hline
 $\vdots$ & $\vdots$  & $\vdots$  \\
\end{tabular}
}
\end{center}
\end{table}

We distinguish rational surface singularities in the sense of Artin
\cite[p.~129]{Artin66} 
from those in the sense of Lipman \cite[Definition~(1.1)]{Lipman69}.
The former is the notion for surfaces defined over an algebraically closed field,
whereas the latter is the notion for excellent surfaces.

Recent works of Cossart and Piltant \cite{CossartPiltant08},
\cite{CossartPiltant09}, Cutkosky \cite{Cutkosky09}
on resolution of singularities of three dimensional 
varieties in arbitrary characteristic 
allow us to define canonical and terminal singularities in the same way as in
characteristic zero \cite{Reid80}. 

\begin{definition}[\cite{HIS13}]
Let $X$ be a quasi-projective normal variety of dimension two or three 
defined over an algebraically 
closed field $k$ of characteristic $p$.
Then $X$ is said to have only canonical (resp. terminal) singularities, 
if the following conditions are satisfied.
\begin{enumerate}
\item [i)]  $X$ is $\mathbf Q$-Gorenstein, i.e. there exists a positive
integer $m$ such that $mK_X$ is a Cartier divisor. 
\item [ii)]  There exists a resolution of singularities 
$\pi :\tilde X\to X$, such that $mK_{\tilde X}\sim\pi^*(mK_X)+
\sum_ia_iE_i$ with $a_i\geq 0$ (resp. $a_i> 0$), where $E:=\cup_iE_i$ is the irreducible 
decomposition of the exceptional divisor of $\pi$.    
\end{enumerate}
\end{definition}
    
As in characteristic zero, the definition of canonical (resp. terminal) singularities 
are independent of the choice of resolutions $\pi :\tilde X\to X$.
\begin{proposition}[HIS,~{\cite[Propositions 11, 12, Lemma 13]{HIS13}}]
Let $X$ be a normal variety of dimension two or three over an 
algebraically closed field $k$ of characteristic $p\geq 0$,
and $\pi : \tilde X\to X$ be a resolution, i.e. 
a proper birational morphism from a nonsingular $\tilde X$. 
Then for any $i>0$, the higher direct image sheaves
$R^i\pi_*\mathcal O_{\tilde X}$ and  
$R^i\pi_*K_{\tilde X}$ are independent of 
the choice of resolutions
$\pi : \tilde X\to X$.
In particular, if $X$ is nonsingular,  
these sheaves are zero.
\end{proposition}

\section{Concrete equations}

In this section we present equations of 
compound Du Val singularities
in characteristic $2$.

\begin{theorem}\label{main}
Let $k$ be an algebraically closed field of characteristic $2$,
$X$ be a hypersurface singularity
defined by one of the following polynomials $f(x,y,z,w)$
as $X:=\Spec k[x,y,z,w]/(f)$.
Then the following assertions hold.
\begin{itemize}
\item [i)] $X$ has a one dimensional singular locus $C$.  
\item [ii)] At any points $x_1, x_2\in C$, the complete local rings
$\hat{\mathcal O}_{X, x_1}$ and $\hat{\mathcal O}_{X, x_2}$ are
isomorphic to each other.  
\item [iii)] $X$ has a crepant resolution which is obtained by 
a succession of blow-ups along singular loci.
More detailed description is given after each equation below.   
\item [iv)] At any point $x\in C$, the complete local ring
$\hat{\mathcal O}_{X, x}$ cannot be expressed as
$k[[x,y,z]]/(g)\hat{\otimes}k[[w]]$ 
with any rational double point $\Spec k[[x,y,z]]/(g)$.
\item [v)] A general hyperplane section $H\subset X$ 
has a rational double point of the type indicated on the left of 
each equation. 
For example, $\mathbf{D_4^1B_1}$ stands for
that $H$ has a rational double point of type $\mathbf{D_4^1}$
in Artin's sense, and $\mathbf{B_1}$ is the type of the rational 
double point in Lipman's sense which appears on $H$ when one views
it as a two dimensional scheme defined over the function field. 
\end{itemize}

{\small

\vspace{2mm}

{\scriptsize $D_4^1:$} 

\noindent
$\mathbf{D_4^1B_1}:$ $z^2+xyz+wx^2+y^3=0$. This can be resolved by a single blow-up.
The exceptional divisor is irreducible and singular.

\noindent
$\mathbf{D_4^1B_2}:$ $z^2+xy^2+xyz+wx^2=0$. This can be resolved by two blow-ups. 
A trivial product of a rational double point of type $A_1$ with a curve appears
after the first blow-up.

\vspace{2mm}

{\scriptsize $D_4^0:$} 

\noindent
\fbox{$\mathbf{D_4^0B_1}:$ $z^2+wx^2+y^3=0$.}
This can be resolved by a single blow-up.
The exceptional divisor is irreducible and singular.

\noindent
\fbox{$\mathbf{D_4^0B_2}:$ $z^2+xy^2+wx^2=0$.} This can be resolved by two blow-ups. After the first blow-up, 
a trivial product of a rational double point of type $A_1$ with a curve appears.  

\vspace{2mm}

{\scriptsize $D_5^1:$} 

\noindent
$\mathbf{D_5^1C_3}:$ $z^2+x^2y+y^2z+xyz+wy^3=0$. This can be resolved by three blow-ups.
After the first blow-up, two disjoint trivial products of
$A_1$ with curves appear.

\vspace{2mm}

{\scriptsize $D_5^0:$} 

\noindent
\fbox{$\mathbf{D_5^0C_3}:$ $z^2+x^2y+y^2z+wy^3=0$.} 
This can be resolved by three blow-ups.
After the first blow-up, two disjoint trivial products of
$A_1$ with curves appear.

\vspace{2mm}

{\scriptsize $D_6^2:$} 
 
\noindent
$\mathbf{D_6^2B_2}:$ $z^2+xy^3+xyz+wx^2+y^5=0$.
This can be resolved by two blow-ups.
After the first blow-up, the singularity $\mathbf {D_4^1B_1}$ 
appears.

\noindent
$\mathbf{D_6^2B_3}:$ $z^2+xy^3+xyz+wx^2=0$.
This can be resolved by three blow-ups.
After the first blow-up, the singularity of type $\mathbf {D_4^1B_2}$ appears.

\noindent
$\mathbf{D_6^2C_4}:$ $z^2+x^2y+xy^3+xyz+wy^4=0$.
This can be resolved by four blow-ups.
After the first blow-up, the singularity of type $\mathbf {D_4^1B_2}$ and the
singularity isomorphic to a trivial product of $A_1$ with a curve appear.

\vspace{2mm}

{\scriptsize $D_6^1:$} 

\noindent
$\mathbf{D_6^1B_2}:$ $ z^2+xy^3+wx^2+y^5=0$.
This can be resolved by two blow-ups.
After the first blow-up, 
the singularity of type $\mathbf{D_4^0B_1}$ appears.

\noindent
\fbox{$\mathbf{D_6^1B_3}:$ $z^2+xy^3+wx^2=0$.} 
This can be resolved by three blow-ups.
After the first blow-up, the singularity of type $\mathbf{D_4^0B_2}$ appears.

\noindent
$\mathbf{D_6^1C_3}:$ $z^2+x^2y+xy^3+xy^2z+wy^3=0$.
This can be resolved by three blow-ups.
After the first blow-up, two disjoint trivial products of
$A_1$ with curves appear.

\noindent
$\mathbf{D_6^1C_4}:$ $z^2+x^2y+xy^3+xy^2z+wy^4=0$.
This can be resolved by three blow-ups.
After the first blow-up, a trivial product of $A_1$ with a curve  
and the singularity of type $\mathbf{D_4^0B_2}$ appear.

\vspace{2mm}

{\scriptsize $D_6^0:$} 

\noindent
$\mathbf{D_6^0C_3}:$ $z^2+x^2y+xy^3+wy^3=0$.
This can be resolved by three blow-ups.
After the first blow-up, two disjoint  trivial products of
$A_1$ with curves appear.

\noindent
$\mathbf{D_6^0C_4}:$ $z^2+x^2y+xy^3+wy^4=0$.
This can be resolved by four blow-ups.
After the first blow-up, the singularities which are isomorphic to a 
trivial product of $A_1$ with a curve and the one of type $\mathbf{D_4^0B_2}$ appear.

\vspace{2mm}

{\scriptsize $D_7^2:$} 

\noindent
$\mathbf{D_7^2B_2}:$ $z^2+xyz+wx^2+y^5=0$.
This can be resolved by two blow-ups.
After the first blow-up, the singularity of type $\mathbf {D_4^1B_1}$ appears.

\noindent
$\mathbf{D_7^2C_5}:$ $z^2+x^2y+y^3z+xyz+wy^5=0$.
This can be resolved by five blow-ups.
After the first blow-up, the singularities which are isomorphic to a 
trivial product of $A_1$ with a curve and the one of type $\mathbf{D_5^1C_3}$ appear.

\vspace{2mm}

{\scriptsize $D_7^1:$} 

\noindent
$\mathbf{D_7^1C_4}:$ $z^2+x^2y+y^3z+xy^2z+wy^4=0$.
This can be resolved by four blow-ups.
After the first blow-up, the singularities which are isomorphic to a 
trivial product of $A_1$ with a curve and the one of type $\mathbf{D_4^0B_2}$ appear.

\noindent
$\mathbf{D_7^1C_5}:$ $z^2+x^2y+y^3z+xy^2z+wy^5=0$.
This can be resolved by five blow-ups.
After the first blow-up, the singularities which are isomorphic to a 
trivial product of $A_1$ with a curve and the one of type $\mathbf{D_5^0C_3}$ appear.

\vspace{2mm}

{\scriptsize $D_7^0:$} 

\noindent
$\mathbf{D_7^0C_4}:$ $z^2+x^2y+y^3z+wy^4=0$.
This can be resolved by four blow-ups.
After the first blow-up, the singularities which are isomorphic to a 
trivial product of $A_1$ with a curve and the one of type $\mathbf{D_4^0B_2}$ appear.

\noindent
$\mathbf{D_7^0C_5}:$ $z^2+x^2y+y^3z+wy^5=0$.
This can be resolved by five blow-ups.
After the first blow-up, a trivial product of $A_1$ with a curve and 
the singularity of type $\mathbf{D_5^0C_3}$ appear.
 
\vspace{2mm}

{\scriptsize $D_8^3:$} 

\noindent
$\mathbf{D_8^3B_3}:$ $ z^2+xy^4+xyz+wx^2+y^7=0$.
This can be resolved by three blow-ups.
After the first blow-up, 
the singularity of type $\mathbf{D_6^2B_2}$ appears.

\noindent
$\mathbf{D_8^3B_4}:$ $z^2+xy^4+xyz+wx^2=0$.
This can be resolved by four blow-ups.
After the first blow-up, the singularity of type $\mathbf{D_6^2B_3}$ appears.

\noindent
$\mathbf{D_8^3C_6}:$ $z^2+x^2y+xy^4+xyz+wy^6=0$.
This can be resolved by six blow-ups.
After the first blow-up, a trivial product of $A_1$ with a curve and 
the singularity of type $\mathbf{D_6^2C_4}$ appear.
 
\vspace{2mm}

{\scriptsize $D_8^2:$} 

\noindent
$\mathbf{D_8^2B_2}:$ $z^2+xy^4+xy^2z+wx^2+y^5=0$.
This can be resolved by two blow-ups.
After the first blow-up, the singularity of type $\mathbf{D_4^0B_1}$ appears.

\noindent
$\mathbf{D_8^2B_3}:$ $ z^2+xy^4+xy^2z+wx^2+y^7=0$.
This can be resolved by three blow-ups.
After the first blow-up, the singularity of type $\mathbf{D_6^1B_2}$ appears.

\noindent
$\mathbf{D_8^2B_4}:$ $ z^2+xy^4+xy^2z+wx^2=0$.
This can be resolved by four blow-ups.
After the first blow-up, 
the singularity of type $\mathbf{D_6^1B_3}$ appears.


\noindent
$\mathbf{D_8^2C_4}:$ $ z^2+x^2y+xy^4+xy^2z+wy^4=0$.
This can be resolved by four blow-ups.
After the first blow-up, a trivial product of $A_1$ with a curve and  
the singularity of type $\mathbf{D_4^0B_2}$ appear.

\noindent
$\mathbf{D_8^2C_5}:$ $ z^2+x^2y+xy^4+xy^2z+wy^5=0$.
This can be resolved by five blow-ups.
After the first blow-up,  a trivial product of $A_1$ with a curve and
the singularity of type $\mathbf{D_6^1C_3}$ appear.

\noindent
$\mathbf{D_8^2C_6}:$ $ z^2+x^2y+xy^4+xy^2z+wy^6=0$.
This can be resolved by six blow-ups.
After the first blow-up,  a trivial product of $A_1$ with a curve and
the singularity of type $\mathbf{D_6^1C_4}$ appear.

\vspace{2mm}

{\scriptsize $D_8^1:$} 

\noindent
$\mathbf{D_8^1B_2}:$ $z^2+xy^4+xy^3z+wx^2+y^5=0$.
This can be resolved by two blow-ups.
After the first blow-up, the singularity of type $\mathbf{D_4^0B_1}$ appears.

\noindent
$\mathbf{D_8^1B_3}:$ $z^2+xy^4+xy^3z+wx^2+y^7=0$.
This can be resolved by three blow-ups,
After the first blow-up, the singularity of type $\mathbf{D_6^1B_2}$ appears.

\noindent
$\mathbf{D_8^1B_4}:$ $z^2+xy^4+xy^3z+wx^2=0$.
This can be resolved by four blow-ups.
After the first blow-up, 
the singularity of type $\mathbf{D_6^1B_3}$ appears.

\noindent
$\mathbf{D_8^1C_4}:$ $z^2+x^2y+xy^4+xy^3z+wy^4=0$.
This can be resolved by four blow-ups.
After the first blow-up, a trivial product of $A_1$ with a curve and
the singularity of type $\mathbf{D_4^0B_2}$ appear.

\noindent
$\mathbf{D_8^1C_5}:$ $z^2+x^2y+xy^4+xy^3z+wy^5=0$.
This can be resolved by five blow-ups.
After the first blow-up, a trivial product of $A_1$ with a curve and
the singularity of type $\mathbf{D_6^0C_3}$ appear.

\noindent
$\mathbf{D_8^1C_6}:$ $z^2+x^2y+xy^4+xy^3z+wy^6=0$.
This can be resolved by six blow-ups.
After the first blow-up, a trivial product of $A_1$ with a curve and
the singularity of type $\mathbf{D_6^0C_4}$ appear.

\vspace{2mm}

{\scriptsize $D_8^0:$} 

\noindent
$\mathbf{D_8^0B_2}:$ $z^2+xy^4+wx^2+y^5=0$.
This can be resolved by two blow-ups.
After the first blow-up, the singularity of type $\mathbf{D_4^0B_1}$ appears.

\noindent
$\mathbf{D_8^0B_3}:$ $z^2+xy^4+wx^2+y^7=0$.
This can be resolved by three blow-ups.
Afer the first blow-up, the singularity of type $\mathbf{D_6^1B_2}$ appears.

\noindent
$\mathbf{D_8^0B_4}:$ $z^2+xy^4+wx^2=0$.
This can be resolved by four blow-ups.
After the first blow-up, the singularity of type $\mathbf{D_6^1B_3}$ appears.

\noindent
\fbox{$\mathbf{D_8^0C_4}:$ $z^2+x^2y+xy^4+wy^4=0$.} 
This can be resolved by four blow-ups.
After the first blow-up, a trivial product of $A_1$ with a curve and
the singularity of type $\mathbf{D_4^0B_2}$ appear.

\noindent
$\mathbf{D_8^0C_5}:$ $z^2+x^2y+xy^4+wy^5=0$.
This can be resolved by two blow-ups.
After the first blow-up, a trivial product of $A_1$ with a curve and
the singularity of type $\mathbf{D_6^0C_3}$ appear.

\noindent
$\mathbf{D_8^0C_6}:$ $z^2+x^2y+xy^4+wy^6=0$.
This can be resolved by two blow-ups.
After the first blow-up, a trivial product of $A_1$ with a curve and
the singularity of type $\mathbf{D_6^0C_4}$ appear.

\vspace{2mm}

{\scriptsize $D_9^3:$} 

\noindent
$\mathbf{D_9^3B_3}:$ $z^2+xyz+wx^2+y^7=0$.
This can be resolved by three blow-ups.
After the first blow-up, the singularity of type $\mathbf{D_7^2B_2}$ appears.

\noindent
$\mathbf{D_9^3C_7}:$ $z^2+x^2y+y^4z+xyz+wy^7=0$.
This can be resolved by seven blow-ups.
After the first blow-up, a trivial product of $A_1$ with a curve and
the singularity of type $\mathbf{D_7^2C_5}$ appear.

\vspace{2mm}
 
{\scriptsize $D_9^2:$} 

\noindent
$\mathbf{D_9^2C_5}:$ $z^2+x^2y+y^4z+xy^2z+wy^5=0$.
This can be resolved by five blow-ups.
After the first blow-up, a trivial product of $A_1$ with a curve and
the singularity of type $\mathbf{D_6^1C_3}$ appear.

\noindent
$\mathbf{D_9^2C_6}:$ $z^2+x^2y+y^4z+xy^2z+wy^6=0$.
This can be resolved by six blow-ups.
After the first blow-up, a trivial product of $A_1$ with a curve and
the singularity of type $\mathbf{D_7^1C_4}$ appear.

\noindent
$\mathbf{D_9^2C_7}:$ $z^2+x^2y+y^4z+xy^2z+wy^7=0$.
This can be resolved by seven blow-ups.
After the first blow-up, a trivial product of $A_1$ with a curve and
the singularity of type $\mathbf{D_7^1C_5}$ appear.

\vspace{2mm}

{\scriptsize $D_9^1:$} 

\noindent
$\mathbf{D_9^1C_5}:$ $z^2+x^2y+y^4z+xy^3z+wy^5=0$.
This can be resolved by five blow-ups.
After the first blow-up, a trivial product of $A_1$ with a curve and
the singularity of type $\mathbf{D_6^0C_3}$ appear.

\noindent
$\mathbf{D_9^1C_6}:$ $z^2+x^2y+y^4z+xy^3z+wy^6=0$.
This can be resolved by six blow-ups.
After the first blow-up, a trivial product of $A_1$ with a curve and
the singularity of type $\mathbf{D_7^0C_4}$ appear.

\noindent
$\mathbf{D_9^1C_7}:$ $z^2+x^2y+y^4z+xy^3z+wy^7=0$.
This can be resolved by seven blow-ups.
After the first blow-up, a trivial product of $A_1$ with a curve and
the singularity of type $\mathbf{D_7^0C_5}$ appear.

\vspace{2mm}
 
{\scriptsize $D_9^0:$} 

\noindent
$\mathbf{D_9^0C_5}:$ $z^2+x^2y+y^4z+wy^5=0$.
This can be resolved by five blow-ups.
After the first blow-up, a trivial product of $A_1$ with a curve and
the singularity of type $\mathbf{D_6^0C_3}$ appear.

\noindent
$\mathbf{D_9^0C_6}:$ $z^2+x^2y+y^4z+wy^6=0$.
This can be resolved by six blow-ups.
After the first blow-up, a trivial product of $A_1$ with a curve and
the singularity of type $\mathbf{D_7^0C_4}$ appear.

\noindent
$\mathbf{D_9^0C_7}:$ $z^2+x^2y+y^4z+wy^7=0$.
This can be resolved by seven blow-ups.
After the first blow-up, a trivial product of $A_1$ with a curve and
the singularity of type $\mathbf{D_7^0C_5}$ appear. 

\noindent
$\qquad\qquad$ $\vdots$

{\scriptsize $D_{2n}^{n-1}:$} 

\noindent
$\mathbf{D_{2n}^{n-1}B_{n-1}}:$ $ z^2+xy^n+xyz+wx^2+y^{2n-1}=0$.
This can be resolved by 
$n-1$ blow-ups.
After the first blow-up, 
the singularity of type $\mathbf{D_{2n-2}^{n-2}B_{n-2}}$ appears.

\noindent
$\mathbf{D_{2n}^{n-1}B_n}:$ $z^2+xy^n+xyz+wx^2=0$.
This can be resolved by $n$ blow-ups.
After the first blow-up, the singularity of type $\mathbf{D_{2n-2}^{n-2}B_{n-1}}$ appears.

\noindent
$\mathbf{D_{2n}^{n-1}C_{2n-2}}:$ $z^2+x^2y+xy^n+xyz+wy^{2(n-1)}=0$.
This can be resolved by $2n-2$ blow-ups.
After the first blow-up, a trivial product of $A_1$ with a curve and
the singularity of type $\mathbf{D_{2n-2}^{n-2}C_{2n-4}}$ appear.
 
\noindent
$\qquad\qquad$ $\vdots$

\vspace{2mm}

{\scriptsize $D_{2n}^{r}: (n/2 < r < n-1)$} 

\noindent
$\mathbf{D_{2n}^{r}B_{r}}:$ $z^2+xy^n+xy^{n-r}z+wx^2+y^{2r+1}=0$.
This can be resolved by $r$ blow-ups.
After the first blow-up, the singularity of type 
$\mathbf{D_{2n-2}^{r-1}B_{r-1}}$ appears.

\noindent
$\mathbf{D_{2n}^{r}B_{r+1}}:$ $z^2+xy^n+xy^{n-r}z+y^{2r+3}+wx^2=0$.
This can be resolved by $r+1$ blow-ups.
After the first blow-up, the singularity of type 
$\mathbf{D_{2n-2}^{r-1}B_{r}}$ appears.

\noindent
$\qquad\qquad$ $\vdots$

\noindent
$\mathbf{D_{2n}^{r}B_{n-1}}:$ $ z^2+xy^n+xy^{n-r}z+wx^2+y^{2n-1}=0$.
This can be resolved by $n-1$ blow-ups.
After the first blow-up, 
the singularity of type $\mathbf{D_{2n-2}^{r-1}B_{n-2}}$ appears.

\noindent
$\mathbf{D_{2n}^{r}B_n}:$ $ z^2+xy^n+xy^{n-r}z+wx^2=0$.
This can be resolved by $n$ blow-ups.
After the first blow-up, 
the singularity of type $\mathbf{D_{2n-2}^{r-1}B_{n-1}}$ appears.

\noindent
$\mathbf{D_{2n}^{r}C_{2r}}:$ $ z^2+x^2y+xy^n+xy^{n-r}z+wy^{2r}=0$.
This can be resolved by $2r$ blow-ups.
After the first blow-up, a trivial product of $A_1$ with a curve and
the singularity of type $\mathbf{D_{2n-2}^{r-1}C_{2r-2}}$ appear.

\noindent
$\mathbf{D_{2n}^{r}C_{2r+1}}:$ $ z^2+x^2y+xy^n+xy^{n-r}z+wy^{2r+1}=0$.
This can be resolved by $2r+1$ blow-ups.
After the first blow-up, a trivial product of $A_1$ with a curve and
the singularity of type $\mathbf{D_{2n-2}^{r-1}C_{2r-1}}$ appear.

\noindent
$\qquad\qquad$ $\vdots$

\noindent
$\mathbf{D_{2n}^{r}C_{2n-3}}:$ $ z^2+x^2y+xy^n+xy^{n-r}z+wy^{2n-3}=0$.
This can be resolved by $2n-3$ blow-ups.
After the first blow-up, a trivial product of $A_1$ with a curve and
the singularity of type $\mathbf{D_{2n-2}^{r-1}C_{2n-5}}$ appear.

\noindent
$\mathbf{D_{2n}^{r}C_{2n-2}}:$ $ z^2+x^2y+xy^n+xy^{n-r}z+wy^{2(n-1)}=0$.
This can be resolved by $2n-2$ blow-ups.
After the first blow-up, a trivial product of $A_1$ with a curve and
the singularity of type $\mathbf{D_{2n-2}^{r-1}C_{2n-4}}$ appear.

\noindent
$\qquad\qquad$ $\vdots$

\vspace{6mm}

{\scriptsize $D_{2n}^{r}:\ (1\leq r\leq n/2)$ }

\noindent
$\mathbf{D_{2n}^{r}B_{\lfloor\frac{n+1}{2}\rfloor}}:$ $z^2+xy^n+xy^{n-r}z+wx^2+y^{2\lfloor\frac{n+1}{2}\rfloor+1}=0$.
This can be resolved by 
$\lfloor\frac{n+1}{2}\rfloor$ blow-ups.
After the first blow-up, 
the singularity of type $\mathbf{D_{2n-2}^{r-1}B_{\lfloor\frac{n-1}{2}\rfloor}}$ 
$($resp. of type $\mathbf{D_{2n-4}^{\max\{0,\, r-2\}}B_{\lfloor\frac{n-1}{2}\rfloor}})$ appears, 
if $n$ is odd $($resp. $n$ is even$)$.
Here the symbol $\lfloor\phantom{11}\rfloor$ stands for rounding down.

\noindent
$\mathbf{D_{2n}^{r}B_{\lfloor\frac{n+1}{2}\rfloor+1}}:$ 
$z^2+xy^n+xy^{n-r}z+wx^2+y^{2\lfloor\frac{n+1}{2}\rfloor+3}=0$.
This can be resolved by 
$\lfloor\frac{n+1}{2}\rfloor+1$ blow-ups.
After the first blow-up, 
the singularity of type $\mathbf{D_{2n-2}^{r-1}B_{\lfloor\frac{n+1}{2}\rfloor}}$ appears.

\noindent
$\phantom{\mathbf{D_{2n}^{n-2}B_{n-2}}:}$ $\vdots$

\noindent
$\mathbf{D_{2n}^{r}B_{n-1}}:$ $ z^2+xy^n+xy^{n-r}z+wx^2+y^{2n-1}=0$.
This can be resolved by $n-1$ blow-ups.
After the first blow-up, 
the singularity of type $\mathbf{D_{2n-2}^{r-1}B_{n-2}}$ appears.

\noindent
$\mathbf{D_{2n}^{r}B_n}:$ $ z^2+xy^n+xy^{n-r}z+wx^2=0$.
This can be resolved by $n$ blow-ups.
After the first blow-up, 
the singularity of type $\mathbf{D_{2n-2}^{r-1}B_{n-1}}$ appears.

\noindent
$\mathbf{D_{2n}^{r}C_n}:$ $ z^2+x^2y+xy^n+xy^{n-r}z+wy^n=0$.
This can be resolved by $n$ blow-ups.
After the first blow-up, a trivial product of $A_1$ with a curve and
the singularity of type $\mathbf{D_{2n-4}^{\max\{0,\,r-2\}}C_{n-2}}$ appear.

\noindent
$\mathbf{D_{2n}^{r}C_{n+1}}:$ $ z^2+x^2y+xy^n+xy^{n-r}z+wy^{n+1}=0$.
This can be resolved by $n+1$ blow-ups.
After the first blow-up, a trivial product of $A_1$ with a curve and
the singularity of type $\mathbf{D_{2n-2}^{r-1}C_{n-1}}$ appear.

\noindent
$\phantom{\mathbf{D_{2n}^{n-2}B_{n-2}}:}$ $\vdots$

\noindent
$\mathbf{D_{2n}^{r}C_{2n-3}}:$ $ z^2+x^2y+xy^n+xy^{n-r}z+wy^{2n-3}=0$.
This can be resolved by $2n-3$ blow-ups.
After the first blow-up, a trivial product of $A_1$ with a curve and
the singularity of type $\mathbf{D_{2n-2}^{r-1}C_{2n-5}}$ appear.

\noindent
$\mathbf{D_{2n}^{r}C_{2n-2}}:$ $ z^2+x^2y+xy^n+xy^{n-r}z+wy^{2n-2}=0$.
This can be resolved by $2n-2$ blow-ups.
After the first blow-up, a trivial product of $A_1$ with a curve and
the singularity of type $\mathbf{D_{2n-2}^{r-1}C_{2n-4}}$ appear.

\noindent
$\phantom{\mathbf{D_{2n}^{n-2}B_{n-2}}:}$ $\vdots$

{\scriptsize $D_{2n}^{0}:$ }

\noindent
$\mathbf{D_{2n}^{0}B_{n/2}}:$ $z^2+xy^n+wx^2+y^{n+1}=0$ with
$n$ even.
This can be resolved by $n/2$ blow-ups.
After the first blow-up, 
the singularity of type 
$\mathbf{D_{2n-4}^0B_{(n-2)/2}}$ appears.

\noindent
$\mathbf{D_{2n}^{0}B_{(n+2)/2}}:$ $z^2+xy^n+wx^2+y^{n+3}=0$ with
$n$ even.
This can be resolved by $(n+2)/2$ blow-ups.
After the first blow-up, 
the singularity of type 
$\mathbf{D_{2n-2}^1B_{n/2}}$ appears.

\noindent
$\phantom{\mathbf{D_{2n}^{n-2}B_{n-2}}:}$ $\vdots$

\noindent
$\mathbf{D_{2n}^{0}B_{n-1}}:$ $ z^2+xy^n+wx^2+y^{2n-1}=0$ with
$n$ even.
This can be resolved by $n-1$ blow-ups.
After the first blow-up, 
the singularity of type $\mathbf{D_{2n-2}^1B_{n-2}}$ appears.

\noindent
$\mathbf{D_{2n}^{0}B_n}:$ $ z^2+xy^n+wx^2=0$ with 
$n$ even.
This can be resolved by $n$ blow-ups.
After the first blow-up, 
the singularity of type $\mathbf{D_{2n-2}^1B_{n-1}}$ appears.

\noindent
$\mathbf{D_{2n}^{0}C_n}:$ $ z^2+x^2y+wy^n=0$.
This can be resolved by $n$ blow-ups.
After the first blow-up, a trivial product of $A_1$ with a curve and
the singularity of type $\mathbf{D_{2n-4}^0C_{n-2}}$ appear.

\noindent
$\mathbf{D_{2n}^{0}C_{n+1}}:$ $ z^2+x^2y+xy^n+wy^{n+1}=0$. 
This can be resolved by $n+1$ blow-ups.
After the first blow-up, a trivial product of $A_1$ with a curve and
the singularity of type $\mathbf{D_{2n-2}^0C_{n-1}}$ appear.

\noindent
$\phantom{\mathbf{D_{2n}^{n-2}B_{n-2}}:}$ $\vdots$

\noindent
$\mathbf{D_{2n}^{0}C_{2n-3}}:$ $ z^2+x^2y+xy^n+wy^{2n-3}=0$.
This can be resolved by $2n-3$ blow-ups.
After the first blow-up, a trivial product of $A_1$ with a curve and
the singularity of type $\mathbf{D_{2n-2}^0C_{2n-5}}$ appear.

\noindent
$\mathbf{D_{2n}^{0}C_{2n-2}}:$ $ z^2+x^2y+xy^n+wy^{2(n-1)}=0$.
This can be resolved by $2n-2$ blow-ups.
After the first blow-up, a trivial product of $A_1$ with a curve and
the singularity of type $\mathbf{D_{2n-2}^0C_{2n-4}}$ appear.

\vspace{2mm}

{\scriptsize $D_{2n+1}^{n-1}:$} 

\noindent
$\mathbf{D_{2n+1}^{n-1}B_{n-1}}:$ $ z^2+xyz+wx^2+y^{2n-1}=0$.
This can be resolved by $n-1$ blow-ups.
After the first blow-up, 
the singularity of type $\mathbf{D_{2n-1}^{n-2}B_{n-2}}$ appears.

\noindent
$\mathbf{D_{2n+1}^{n-1}C_{2n-1}}:$ $ z^2+x^2y+y^nz+xyz+wy^{2n-1}=0$.
This can be resolved by $2n-1$ blow-ups.
After the first blow-up, a trivial product of $A_1$ with a curve and
the singularity of type $\mathbf{D_{2n-1}^{n-2}C_{2n-3}}$ appear.

\newpage 

{\scriptsize $D_{2n+1}^{r}:$ $(n/2< r <n-1)$.}

\noindent
$\mathbf{D_{2n+1}^{r}C_{2r+1}}:$ $ z^2+x^2y+y^nz+xy^{n-r}z+wy^{2r+1}=0$.
This can be resolved by $2r+1$ blow-ups.
After the first blow-up, a trivial product of $A_1$ with a curve and
the singularity of type $\mathbf{D_{2n-1}^{r-1}C_{2r-1}}$ appear.

\noindent
$\phantom{\mathbf{D_{2n+1}^{n-2}B_{n-2}}:}$ $\vdots$

\noindent
$\mathbf{D_{2n+1}^{r}C_{2n-2}}:$ $ z^2+x^2y+y^nz+xy^{n-r}z+wy^{2n-2}=0$.
This can be resolved by $2n-2$ blow-ups.
After the first blow-up, a trivial product of $A_1$ with a curve and
the singularity of type $\mathbf{D_{2n-1}^{r-1}C_{2n-4}}$ appear.

\noindent
$\mathbf{D_{2n+1}^{r}C_{2n-1}}:$ $ z^2+x^2y+y^nz+xy^{n-r}z+wy^{2n-1}=0$.
This can be resolved by $2n-1$ blow-ups.
After the first blow-up, a trivial product of $A_1$ with a curve and
the singularity of type $\mathbf{D_{2n-1}^{r-1}C_{2n-3}}$ appear.

\vspace{2mm}

{\scriptsize  $D_{2n+1}^{r}:$ $(1\leq r \leq n/2)$.}

\noindent
$\mathbf{D_{2n+1}^{r}C_{n+1}}:$ $ z^2+x^2y+y^nz+xy^{n-r}z+wy^{n+1}=0$.
This can be resolved by $n+1$ blow-ups.
After the first blow-up, a trivial product of $A_1$ with a curve and
the singularity of type $\mathbf{D_{2n-2}^{r-1}C_{n-1}}$ appear.

\noindent
$\mathbf{D_{2n+1}^{r}C_{n+2}}:$ $ z^2+x^2y+y^nz+xy^{n-r}z+wy^{n+2}=0$.
This can be resolved by $n+2$ blow-ups.
After the first blow-up, a trivial product of $A_1$ with a curve and
the singularity of type $\mathbf{D_{2n-1}^{r-1}C_{n}}$ appear.

\noindent
$\phantom{\mathbf{D_{2n+1}^{n-2}B_{n-2}}:}$ $\vdots$

\noindent
$\mathbf{D_{2n+1}^{r}C_{2n-2}}:$ $ z^2+x^2y+y^nz+xy^{n-r}z+wy^{2n-2}=0$.
This can be resolved by $2n-2$ blow-ups.
After the first blow-up, a trivial product of $A_1$ with a curve and
the singularity of type $\mathbf{D_{2n-1}^{r-1}C_{2n-4}}$ appear.

\noindent
$\mathbf{D_{2n+1}^{r}C_{2n-1}}:$ $ z^2+x^2y+y^nz+xy^{n-r}z+wy^{2n-1}=0$.
This can be resolved by $2n-1$ blow-ups.
After the first blow-up, a trivial product of $A_1$ with a curve and
the singularity of type $\mathbf{D_{2n-1}^{r-1}C_{2n-3}}$ appear.

{\scriptsize $D_{2n+1}^{0}:$} 

\noindent
$\mathbf{D_{2n+1}^{0}C_{n+1}}:$ $ z^2+x^2y+y^nz+wy^{n+1}=0$.
This can be resolved by $n+1$ blow-ups.
After the first blow-up, a trivial product of $A_1$ with a curve and
the singularity of type $\mathbf{D_{2n-2}^0C_{n-1}}$ appear.

\noindent
$\mathbf{D_{2n+1}^{0}C_{n+2}}:$ $ z^2+x^2y+y^nz+wy^{n+2}=0$.
This can be resolved by $n+2$ blow-ups.
After the first blow-up, a trivial product of $A_1$ with a curve and
the singularity of type $\mathbf{D_{2n-1}^0C_{n}}$ appear.

\noindent
$\phantom{\mathbf{D_{2n+1}^{n-2}B_{n-2}}:}$ $\vdots$

\noindent
$\mathbf{D_{2n+1}^{0}C_{2n-2}}:$ $ z^2+x^2y+y^nz+wy^{2n-2}=0$.
This can be resolved by $2n-2$ blow-ups.
After the first blow-up, a trivial product of $A_1$ with a curve and
the singularity of type $\mathbf{D_{2n-1}^0C_{2n-4}}$ appear.

\noindent
$\mathbf{D_{2n+1}^{0}C_{2n-1}}:$ $ z^2+x^2y+y^nz+wy^{2n-1}=0$.
This can be resolved by $2n-1$ blow-ups.
After the first blow-up, a trivial product of $A_1$ with a curve and
the singularity of type $\mathbf{D_{2n-1}^0C_{2n-3}}$ appear.

\vspace{2mm}
{\scriptsize $E_7^3:$} 

\noindent
$\mathbf{E_7^3F_4}:$ $ z^2+x^3+xy^3+xyz+wy^4=0$.
This can be resolved by four blow-ups.
After the first blow-up, 
the singularity of type $\mathbf{D_6^2B_3}$ appears.

\vspace{2mm}
{\scriptsize $E_7^2:$} 

\noindent
$\mathbf{E_7^2F_4}:$ $ z^2+x^3+xy^3+y^3z+wy^4=0$.
This can be resolved by three blow-ups.
After the first blow-up, 
the singularity of type $\mathbf{D_6^1B_3}$ appears.

\vspace{2mm}
{\scriptsize $E_7^1:$} 

\noindent
$\mathbf{E_7^1F_4}:$ $ z^2+x^3+xy^3+wy^4=0$.
This can be resolved by four blow-ups.
After the first blow-up, 
the singularity of type $\mathbf{D_6^1B_3}$ appears.

\vspace{2mm}
{\scriptsize $E_8^3:$} 

\noindent
$\mathbf{E_8^3F_4}:$ $ z^2+x^3+y^3z+wy^4=0$.
This can be resolved by four blow-ups.
After the first blow-up, 
the singularity of type $\mathbf{D_6^1B_3}$ appears.

{\scriptsize $E_8^2:$} 

\noindent
$\mathbf{E_8^2F_4}:$ $ z^2+x^3+xy^2z+wy^4=0$.
This can be resolved by three blow-ups.
After the first blow-up, 
the singularity of type $\mathbf{D_6^1B_3}$ appears.

\vspace{2mm}
{\scriptsize $E_8^1:$} 

\noindent
$\mathbf{E_8^1F_4}:$ $ z^2+x^3+xy^3z+wy^4=0$.
This can be resolved by four blow-ups.
After the first blow-up, 
the singularity of type $\mathbf{D_6^1B_3}$ appears.

{\scriptsize $E_8^0:$} 

\noindent
\fbox{$\mathbf{E_8^0F_4}:$ $ z^2+x^3+wy^4=0$.} 
This can be resolved by four blow-ups.
After the first blow-up, 
the singularity of type $\mathbf{D_6^1B_3}$ appears.
}

\end{theorem}

\begin{proof}
We divide the set of equations in question into three subsets, 
i.e.  the subset consisting of equations labeled as
type $D_4^1$ through $D_7^0$ (cf. Table 3), 
the subset consisting of equations of type 
$D_{2n}^{n-1}$ through $D_{2n+1}^{0}$ with $n\geq 4$ (cf. Table 4)
and thirdly the subset consisting of equations labeled as type 
$E_7^3$ through $E_8^0$ (cf. Table~2).
We prove the assertion for type $D_N$ by induction on $N\geq 4$, 
then one by one for type $E_7$, $E_8$.

As the proof of i), ii), iii) and iv) consists of repetition of  
elementary procedures,
i.e.  blowing up, applying Jacobian criterion for regularity, 
exchanging coordinates, etc., 
we present here only the essential step of induction for 
assertion iii) and iv),
i.e.  type $D_{2n}^{n-1}$ through $D_{2n+1}^0$ with $n\geq 4$.

By Jacobian criterion for regularity over
an algebraically closed field $k$, 
we find that the line defined by 
$x=y=z=0$ is the singular locus of 
each hypersurface.
So we blow up this line and examine the resulting singularities 
first in the chart 
$x'=x/y$, $y'=y$, $z'=z/y$,
($w$ remains unchanged).
The equations are obtained easily (cf. the tables below
which should be combined with Table 4), 
so we identify the equations using the induction hypothesis.
Verification goes straightforward, 
however some coordinate changes are required
in the following cases.

For the singularity of type $\mathbf{D_{2n}^{r}B_{\lfloor (n+1)/2 \rfloor}}$ 
with $1\leq r \leq n/2$ and $4\leq n$,
we have the equation after a blow-up
$(z')^2+x'(y')^{n-1}+x'(y')^{n-r}z'+w(x')^2+(y')^{2\lfloor\frac{n+1}{2}\rfloor-1}=0$.
First we consider the case $n$ is even.
Introduce a new coordinate
$y'':=y'+x'$, then the equation is
$(z')^2+x'(x'+y'')^{n-1}+x'(x'+y'')^{n-r}z'+w(x')^2+(x'+y'')^{n-1}=0$.
Then factor out $(x')^2$ and obtain 
$(z')^2+x'(y'')^{n-r}z'+x'(y'')^{n-2}
+u_1(y'')^{n-1}+w'(x')^2=0$ with 
$u_1$ a unit in $k[[x', y'', z', w]]$ and
an appropriate coordinate $w'$.
If $r>2$, we multiply
$x', y'', z', w'$ by some units and may replace
$u_1$ by $1$.
The resulting singularity is of type 
$\mathbf{D_{2(n-2)}^{r-2}B_{\lfloor\frac{(n-2)+1}{2}\rfloor}}$.
If $r\leq 2$, we put $u_2:=1+(y'')^{2-r}z'$ and have the equation
$(z')^2+u_2x'(y'')^{n-2}
+u_1(y'')^{n-1}+w'(x')^2=0$ with $u_1, u_2$ units.
We may replace $u_1, u_2$ by $1$ in a similar way as above. 
This is the singularity of type 
$\mathbf{D_{2(n-2)}^{0}B_{(n-2)/2}}$.
Secondly we consider the case $n$ is odd.
But a straightforward checking finds 
the singularity is of type 
$\mathbf{D_{2(n-1)}^{r-1}B_{\lfloor\frac{(n-1)+1}{2}\rfloor}}$.

For the singularity of type $\mathbf{D_{2n}^{r}C_{n}}$ 
with $1\leq r \leq n/2$ and $4\leq n$,
after a blow-up, we have the hypersurface  
$(z')^2+(x')^2{y'}+x'(y')^{n-1}+{x'}(y')^{n-r}{z'}+w(y')^{n-2}=0$.
If $r>2$, this is the singularity of type 
$\mathbf{D_{2(n-2)}^{r-2}C_{n-2}}$.  
If $r\leq 2$, we introduce a new coordinate
$w':=w+{x'}(y')^{2-r}{z'}$ and have  
$(z')^2+(x')^2{y'}+x'(y')^{n-1}+w'(y')^{n-2}=0$.
This is the singularity of type $\mathbf{D_{2(n-2)}^{0}C_{n-2}}$
if $n\geq 5$, and of type $\mathbf{D_{4}^{0}B_{2}}$ if $n=4$.

For $\mathbf{D_{2n}^{0}B_{n/2}}$ with $4\leq n$ even,
the hypersurface after a blow-up is 
$(z')^2+{x'}(y')^{n-1}+{y'}^{n-1}+w(x')^2=0$.
Another new coordinate $y'':=y'+x'$ gives
$(z')^2+{x'}{(y''+x')}^{n-1}+{(y''+x')}^{n-1}+w(x')^2=0$.
Then factoring out $(x')^2$, 
we have 
$(z')^2+u_1{x'}(y'')^{n-2}+(y'')^{n-1}+{w'}(x')^2=0$
with an appropriate coordinate $w'$
and a unit $u_1$.
This $u_1$ may be replaced by $1$ and the singularity is of type 
$\mathbf{D_{2(n-2)}^0B_{(n-2)/2}}$.

{
\bigskip\tiny 
\begin{tabular}{l|l|l|l}
\hline & \textrm{Original} & 
\textrm{Equation after a blow-up $x'=x/y$, $y'=y$, $z'=z/y$.}&
\textrm{Resulting} \\ 
\hline 
$D_{2n}^{r}$ 
& $\mathbf{D_{2n}^{r}B_{r}}$ &
${y'}^2[{z'}^2+{x'}{y'}^{n-1}+{x'}{y'}^{n-r}{z'}+w{x'}^2+{y'}^{2r-1}]=0$
& $\mathbf{D_{2n-2}^{r-1}B_{r-1}}$ \\ 
& $\mathbf{D_{2n}^{r}B_{r+1}}$ &
${y'}^2[{z'}^2+{x'}{y'}^{n-1}+{x'}{y'}^{n-r}{z'}+w{x'}^2+{y'}^{2r+1}]=0$
& $\mathbf{D_{2n-2}^{r-1}B_{r}}$\\ 
& $\quad\vdots$ & $\quad\vdots$ & $\quad\vdots$ \\ 
& $\mathbf{D_{2n}^{r}B_{n-1}}$ &
${y'}^2[{z'}^2+{x'}{y'}^{n-1}+{x'}{y'}^{n-r}{z'}+w{x'}^2+{y'}^{2n-3}]=0$
& $\mathbf{D_{2n-2}^{r-1}B_{n-2}}$\\ 
& $\mathbf{D_{2n}^{r}B_{n}}$ &
${y'}^2[{z'}^2+{x'}{y'}^{n-1}+{x'}{y'}^{n-r}{z'}+w{x'}^2]=0$
& $\mathbf{D_{2n-2}^{r-1}B_{n-1}}$\\ 
& $\mathbf{D_{2n}^{r}C_{2r}}$ &
${y'}^2[{z'}^2+{x'}^2{y'}+{x'}{y'}^{n-1}+{x'}{y'}^{n-r}{z'}+w{y'}^{2r-2}]=0$
& $\mathbf{D_{2n-2}^{r-1}C_{2r-2}}$\\ 
& $\mathbf{D_{2n}^{r}C_{2r+1}}$ &
${y'}^2[{z'}^2+{x'}^2{y'}+{x'}{y'}^{n-1}+{x'}{y'}^{n-r}{z'}+w{y'}^{2r-1}]=0$
& $\mathbf{D_{2n-2}^{r-1}C_{2r-1}}$\\ 
& $\quad\vdots$ & $\quad\vdots$ & $\quad\vdots$\\ 
& $\mathbf{D_{2n}^{r}C_{2n-2}}$ &
${y'}^2[{z'}^2+{x'}^2{y'}+{x'}{y'}^{n-1}+{x'}{y'}^{n-r}{z'}+w{y'}^{2n-4}]=0$
& $\mathbf{D_{2n-2}^{r-1}C_{2n-4}}$\\ 
\cline{2-4} $\quad\vdots$ & $\quad\vdots$ & $\quad\vdots$ & \\ 
\cline{2-4}
$D_{2n}^{r}$ & $\mathbf{D_{2n}^{r}B_{\lfloor (n+1)/2\rfloor}}$ &
${y'}^2[{z'}^2+{x'}{y'}^{n-1}+{x'}{y'}^{n-r}{z'}+w{x'}^2+{y'}^{2\lfloor(n+1)/2\rfloor-1}]=0$
& (*) \\ 
& $\mathbf{D_{2n}^{r}B_{\lfloor
    (n+1)/2\rfloor+1}}$ &
${y'}^2[{z'}^2+{x'}{y'}^{n-1}+{x'}{y'}^{n-r}{z'}+w{x'}^2+{y'}^{2\lfloor
  (n+1)/2\rfloor+1}]=0$ & (**) \\ 
& $\quad\vdots$ & $\quad\vdots$ & $\quad\vdots$ \\ 
& $\mathbf{D_{2n}^{r}B_{n-1}}$ &
${y'}^2[{z'}^2+{x'}{y'}^{n-1}+{x'}{y'}^{n-r}{z'}+w{x'}^2+{y'}^{2n-3}]=0$
& $\mathbf{D_{2n-2}^{r-1}B_{n-2}}$ \\ 
& $\mathbf{D_{2n}^{r}B_{n}}$ 
& ${y'}^2[{z'}^2+{x'}{y'}^{n-1}+{x'}{y'}^{n-r}{z'}+w{x'}^2]=0$
& $\mathbf{D_{2n-2}^{r-1}B_{n-1}}$\\ 
& $\mathbf{D_{2n}^{r}C_{n}}$ &
${y'}^2[{z'}^2+{x'}^2{y'}+{x'}{y'}^{n-1}+{x'}{y'}^{n-r}{z'}+w{y'}^{n-2}]=0$
& (***) \\ 
& $\mathbf{D_{2n}^{r}C_{n+1}}$ &
${y'}^2[{z'}^2+{x'}^2{y'}+{x'}{y'}^{n-1}+{x'}{y'}^{n-r}{z'}+w{y'}^{n-1}]=0$
& $\mathbf{D_{2(n-1)}^{r-1}C_{n-1}}$ \\ 
& $\quad\vdots$ & $\quad\vdots$ & $\quad\vdots$ \\ 
& $\mathbf{D_{2n}^{r}C_{2n-3}}$ &
${y'}^2[{z'}^2+{x'}^2{y'}+{x'}{y'}^{n-1}+{x'}{y'}^{n-r}{z'}+w{y'}^{2(n-1)-3}]=0$
& $\mathbf{D_{2(n-1)}^{r-1}C_{2n-5}}$\\ 
& $\mathbf{D_{2n}^{r}C_{2n-2}}$ &
${y'}^2[{z'}^2+{x'}^2{y'}+{x'}{y'}^{n-1}+{x'}{y'}^{n-r}{z'}+w{y'}^{2(n-1)-2}]=0$
& $\mathbf{D_{2(n-1)}^{r-1}C_{2n-4}}$\\ 
\cline{2-4} $\quad\vdots$ & $\quad\vdots$ & $\quad\vdots$ & \\ 
\cline{2-4}
$D_{2n}^0$ 
& $\mathbf{D_{2n}^0B_{n/2}}$ 
& ${y'}^2[{z'}^2+{x'}{y'}^{n-1}+w{x'}^2+{y'}^{n-1}]=0$ &
$\mathbf{D_{2n-4}^0B_{(n-2)/2}}$ \\ 
& $\mathbf{D_{2n}^0B_{(n+2)/2}}$ &
${y'}^2[{z'}^2+{x'}{y'}^{n-1}+w{x'}^2+{y'}^{n+1}]=0$ & 
$\mathbf{D_{2n-2}^1B_{n/2}}$ \\ 
& $\quad\vdots$ & $\quad\vdots$ & $\quad\vdots$ \\ 
& $\mathbf{D_{2n}^0B_{n-1}}$ &
${y'}^2[{z'}^2+{x'}{y'}^{n-1}+w{x'}^2+{y'}^{2n-3}]=0$ & 
$\mathbf{D_{2n-2}^1B_{n-2}}$ \\ 
& $\mathbf{D_{2n}^0B_{n}}$ &
${y'}^2[{z'}^2+{x'}{y'}^{n-1}+w{x'}^2]=0$ & 
$\mathbf{D_{2n-2}^1B_{n-1}}$ \\ 
& $\mathbf{D_{2n}^0C_{n}}$ &
${y'}^2[{z'}^2+{x'}^2{y'}+{x'}{y'}^{n-1}+w{y'}^{n-2}]=0$ & $\mathbf{D_{2n-4}^0C_{n-2}}$ \\ 
& $\mathbf{D_{2n}^0C_{n+1}}$ &
${y'}^2[{z'}^2+{x'}^2{y'}+{x'}{y'}^{n-1}+w{y'}^{n-1}]=0$ & $\mathbf{D_{2n-2}^0C_{n-1}}$\\ 
& $\quad\vdots$ & $\quad\vdots$ & $\quad\vdots$ \\ 
& $\mathbf{D_{2n}^0C_{2n-3}}$ &
${y'}^2[{z'}^2+{x'}^2{y'}+{x'}{y'}^{n-1}+w{y'}^{2n-5}]=0$&$\mathbf{D_{2n-2}^0C_{2n-5}}$\\ 
& $\mathbf{D_{2n}^0C_{2n-2}}$ &
${y'}^2[{z'}^2+{x'}^2{y'}+{x'}{y'}^{n-1}+w{y'}^{2n-4}]=0$&$\mathbf{D_{2n-2}^0C_{2n-4}}$\\ 
\hline 
\end{tabular}
\begin{flushleft}
{\tiny
(*) $\mathbf{D_{2(n-2)}^{\max\{0, r-2\}}B_{\lfloor\frac{n-1}{2}\rfloor}}$ if $n$ is even, 
$\mathbf{D_{2(n-1)}^{r-1}B_{\lfloor\frac{n-1}{2}\rfloor}}$ if $n$ is odd. \\
(**)
$\mathbf{D_{2(n-1)}^1B_{n-1}}$ if $n$ is even and $r=1$,
$\mathbf{D_{2(n-1)}^{r-1}B_{\lfloor\frac{n+1}{2}\rfloor}}$ otherwise.   \\
(***) $\mathbf{D_{2(n-2)}^{\max\{0, r-2\}}C_{n-2}}$ if $n\geq 5$, 
$\mathbf{D_{4}^{0}B_{2}}$ if $n=4$. 
}
\end{flushleft}
}

For $\mathbf{D_{2n}^{0}B_{(n+2)/2}}$ with $4\leq n$ even,
we have $(z')^2+{x'}(y')^{n-1}+(y')^{n+1}+w(x')^2=0$ after a blow-up.
Then introduce a new coordinate 
$y'':=y'+z'$,  we have
$(z')^2+{x'}(y''+z')^{n-1}+(y''+z')^{n+1}+w(x')^2=0$.
Then factor out $(z')^2$ and we have
$u_1(z')^2+{x'}(y'')^{n-1}+{x'}(y'')^{n-2}z'+(y'')^{n+1}+(y'')^{n}z'+w(x')^2=0$
with a unit $u_1$ in $k[[x', y'', z', w]]$.
Another new coordinate
$x'':=x'+y''z'$ makes the equation
$u_2(z')^2+{x''}(y'')^{n-1}+{x''}(y'')^{n-2}z'+(y'')^{n+1}+w(x'')^2=0$
with a unit $u_2$.
This unit $u_2$ may be replaced by $1$ by multiplying  
$x'', y'', z', w$ by some units
and we have the singularity of type 
$\mathbf{D_{2(n-1)}^1B_{n/2}}$.

$\quad\vdots$

For $\mathbf{D_{2n}^{0}B_{n-1}}$ with $4\leq n$ even,
we have  $(z')^2+x'(y')^{n-1}+(y')^{2n-3}+w(x')^2=0$ after a blow-up.
Then introduce new coordinates as in  
$\mathbf{D_{2n}^{0}B_{(n+2)/2}}$ above, we have
$u_2(z')^2+{x''}(y'')^{n-1}+{x''}(y'')^{n-2}z'+(y'')^{2n-3}+w(x'')^2=0$
with a unit $u_2$.
Multiplying  
$x'', y'', z', w$ by some units,
we replace $u_2$ by $1$ and have the singularity of type 
$\mathbf{D_{2(n-1)}^1B_{(n-1)-1}}$.

For $\mathbf{D_{2n}^{0}B_n}$ with $4\leq n$ even,
we have $ (z')^2+x'(y')^{n-1}+w(x')^2=0$ after a blow-up.
Then introduce a new coordinate $y'':=y'+z'$, we have
${z'}^2+{x'}{(y''+z')}^{n-1}+w(x')^2=0$.
Then factoring out $(z')^2$, 
we have 
$u_1(z')^2+{x'}(y'')^{n-1}+x'(y'')^{n-2}z'+{w}(x')^2=0$
with a unit $u_1$.
Then multiplying  
$x', y'', z', w$ by some units again,
we may replace $u_1$ by $1$ and have the singularity of type 
$\mathbf{D_{2(n-1)}^1B_{n-1}}$.

For $\mathbf{D_{2n}^{0}C_{n}}$ with $4 \leq n$,
we have $(z')^2+(x')^2y'+x'(y')^{n-1}+w(y')^{n-2}=0$ after a blow-up.
A new coordinate $w':=w+x'y'+x'$ makes the equation
$(z')^2+(x')^2y'+x'(y')^{n-2}+w'(y')^{n-2}=0$. 
This is the singularity of type $\mathbf{D_{2(n-2)}^0C_{n-2}}$.

{
\bigskip
\tiny 
\begin{tabular}{l|l|l|l}
\hline & \textrm{Original} & 
\textrm{Equation after a blow-up $x'=x/y$, $y'=y$, $z'=z/y$.}&
\textrm{Resulting} \\ 
\hline 
$D_{2n+1}^{n-1}$ 
& $\mathbf{D_{2n+1}^{n-1}B_{n-1}}$ &
${y'}^2[{z'}^2+{x'}{y'}{z'}+w{x'}^2+{y'}^{2n-3}]=0$ & 
$\mathbf{D_{2n-1}^{n-2}B_{n-2}}$ \\ 
& $\mathbf{D_{2n+1}^{n-1}C_{2n-1}}$ &
${y'}^2[{z'}^2+{x'}^2{y'}+{y'}^{n-1}{z'}+{x'}{y'}{z'}+w{y'}^{2n-3}]=0$
& $\mathbf{D_{2n-1}^{n-2}C_{2n-3}}$\\ 
\cline{2-4} $\quad\vdots$ & $\quad\vdots$ & $\quad\vdots$ &  \\ 
\cline{2-4}
$D_{2n+1}^r$ 
& $\mathbf{D_{2n+1}^rC_{2r+1}}$ &
${y'}^2[{z'}^2+{x'}^2{y'}+{x'}{y'}^{n-r}{z'}+{y'}^{n-1}{z'}+w{y'}^{2r-1}]=0$
& $\mathbf{D_{2n-1}^{r-1}C_{2r-1}}$ \\ 
& $\mathbf{D_{2n+1}^rC_{2r+2}}$ &
${y'}^2[{z'}^2+{x'}^2{y'}+{x'}{y'}^{n-r}{z'}+{y'}^{n-1}{z'}+w{y'}^{2r}]=0$
& $\mathbf{D_{2n-1}^{r-1}C_{2r}}$ \\ 
& $\quad\vdots$ & $\quad\vdots$ & $\quad\vdots$ \\ 
& $\mathbf{D_{2n+1}^rC_{2n-2}}$ &
${y'}^2[{z'}^2+{x'}^2{y'}+{x'}{y'}^{n-r}{z'}+{y'}^{n-1}{z'}+w{y'}^{2n-4}]=0$ & 
$\mathbf{D_{2n-1}^{r-1}C_{2n-4}}$ \\ 
& $\mathbf{D_{2n+1}^rC_{2n-1}}$ &
${y'}^2[{z'}^2+{x'}^2{y'}+{x'}{y'}^{n-r}{z'}+{y'}^{n-1}{z'}+w{y'}^{2n-3}]=0$
& $\mathbf{D_{2n-1}^{r-1}C_{2n-3}}$ \\ 
\cline{2-4} 
$\quad\vdots$ & $\quad\vdots$ & $\quad\vdots$ & \\ 
\cline{2-4}
$D_{2n+1}^r$ & $\mathbf{D_{2n+1}^rC_{n+1}}$ &
${y'}^2[{z'}^2+{x'}^2{y'}+{x'}{y'}^{n-r}{z'}+{y'}^{n-1}{z'}+w{y'}^{n-1}]=0$
&  $\mathbf{D_{2n-2}^{r-1}C_{n-1}}$ \\ 
& $\mathbf{D_{2n+1}^rC_{n+2}}$ &
${y'}^2[{z'}^2+{x'}^2{y'}+{x'}{y'}^{n-r}{z'}+{y'}^{n-1}{z'}+w{y'}^{n}]=0$
& $\mathbf{D_{2n-1}^{r-1}C_{n}}$ \\ 
& $\quad\vdots$ & $\quad\vdots$ & $\quad\vdots$\\ 
& $\mathbf{D_{2n+1}^rC_{2n-2}}$ &
${y'}^2[{z'}^2+{x'}^2{y'}+{x'}{y'}^{n-r}{z'}+{y'}^{n-1}{z'}+w{y'}^{2n-4}]=0$
& $\mathbf{D_{2n-1}^{r-1}C_{2n-4}}$ \\ 
& $\mathbf{D_{2n+1}^rC_{2n-1}}$ &
${y'}^2[{z'}^2+{x'}^2{y'}+{x'}{y'}^{n-r}{z'}+{y'}^{n-1}{z'}+w{y'}^{2n-3}]=0$
& $\mathbf{D_{2n-1}^{r-1}C_{2n-3}}$ \\ 
\cline{2-4} $\quad\vdots$ & $\quad\vdots$ & $\quad\vdots$ & \\ 
\cline{2-4}
$D_{2n+1}^0$ 
& $\mathbf{D_{2n+1}^0C_{n+1}}$ &
${y'}^2[{z'}^2+{x'}^2{y'}+{y'}^{n-1}{z'}+w{y'}^{n-1}]=0$ &
$\mathbf{D_{2n-2}^0C_{n-1}}$\\ 
& $\mathbf{D_{2n+1}^0C_{n+2}}$ &
${y'}^2[{z'}^2+{x'}^2{y'}+{y'}^{n-1}{z'}+w{y'}^{n}]=0$ &
$\mathbf{D_{2n-1}^0C_{n}}$\\ 
& $\quad\vdots$ & $\quad\vdots$ & $\quad\vdots$\\ 
& $\mathbf{D_{2n+1}^0C_{2n-2}}$ &
${y'}^2[{z'}^2+{x'}^2{y'}+{y'}^{n-1}{z'}+w{y'}^{2n-4}]=0$ &
$\mathbf{D_{2n-1}^0C_{2n-4}}$ \\ 
& $\mathbf{D_{2n+1}^0C_{2n-1}}$ &
${y'}^2[{z'}^2+{x'}^2{y'}+{y'}^{n-1}{z'}+w{y'}^{2n-3}]=0$ &
$\mathbf{D_{2n-1}^0C_{2n-3}}$ \\ 
\hline 
\end{tabular}
\bigskip
}

For $\mathbf{D_{2n+1}^{r}C_{n+1}}$ with $1\leq r \leq n/2$ and $4\leq n$.
We have  $(z')^2+(x')^2y'+x'(y')^{n-r}z'+(y')^{n-1}z+w(y')^{n-1}=0$.
A new coordinate $w':=w+z'$ makes the equation
$(z')^2+(x')^2y'+x'(y')^{n-r}z'+w(y')^{n-1}=0$.
This is the singularity of type $\mathbf{D_{2(n-1)}^{r-1}C_{n-1}}$.

For the singularity of type $\mathbf{D_{2n+1}^{0}C_{n+1}}$ with 
$1\leq r \leq n/2$ and $4\leq n$.
We have 
$ (z')^2+(x')^2y'+(y')^{n-1}z'+w(y')^{n-1}=0$.
A new coordinate $w':=w+z'$ makes the equation
$ (z')^2+(x')^2y'+w'(y')^{n-1}=0$.
This is the singularity of type $\mathbf{D_{2(n-1)}^0C_{n-1}}$.

In the other chart $x''=x$, $y''=y/x$, $z''=z/x$,
($w$ remains unchanged),
for the equations labeled with $\mathbf{C_n}$ 
we find a trivial product of the rational double point
of type $A_1$ and a nonsingular curve, 
but for other equations the resulting surfaces are 
nonsingular.
This makes the induction work, and we obtain assertions iii) and iv).

For the assertion v), we consider a generic hyperplane section and substitute 
$w$ by $\xi x+\eta y+\theta z+\iota$
with variables $\xi, \eta, \theta, \iota$.
Each polynomial listed in Tables 2, 3, 4 is viewed as an element of  
$k(\xi, \eta, \theta, \iota)[x,y,z]$.
Then we apply the classification algorithm 
of Lipman \cite[\S 24]{Lipman69} first over the function field
$k(\xi, \eta, \theta, \iota)$, then we examine how the field 
extension of $k(\xi, \eta, \theta, \iota)$
affects the singularities.  
If the invariant $\tau$ of the quadratic form  $Q(X,Y,Z)$ 
(loc. cit.) satisfies $\tau\geq 2$,  
we have a generic hyperplane section with the singularity of type $B_n$.
If $\tau=1$,  
then we consider the cubic form
$\mathrm{\overline{G}(U,V)}\in k(\xi, \eta, \theta, \iota)[U, V]$ (loc. cit.).
If 
$\mathrm{\overline{G}(U,V)}$ is the product of a linear
and an irreducible quadratic factor over $k(\xi, \eta, \theta, \iota)$, 
we have the singularity of type $C_3$ .
If $\mathrm{\overline{G}(U,V)=UV^2}$,
we have type $C_n$ with $n\geq 4$,
and if $\mathrm{\overline{G}(U,V)}=a\mathrm{V^3}$ $(a\ne 0)$,
we have type $F_4$.
For singularities over an algebraically closed field
$\overline{k(\xi, \eta, \theta, \iota)}$,
we need examine if the quadratic form $Q(X,Y,Z)$ as well as the cubic form
 $\mathrm{\overline{G}(U,V)}$ are factorized further.
\end{proof}

\begin{remarks}
i) Note that when we consider a three dimensional singularity,
taking general hyperplane sections
and blowing up singular loci are not compatible in general.

ii) We substitute $w$ with $\xi x+\eta y+\theta z+\iota$
and consider each equation 
$f(x,y,z,\xi z+\eta y+\theta z +\iota)=0$
given by an element of
$k(\xi,\eta, \theta, \iota)[x,y,z]$.
Then the Jacobian ideal is given by  
$$
(\partial F/\partial x,\partial F/\partial y, \partial F/\partial z, F)\subset k(\xi,\eta, \theta, \iota)[x,y,z]
$$
with $F:=f(x,y,z,\xi z+\eta y+\theta z +\iota)$ and the derivations are
considered as elements of $\mathrm{Der}_{k(\xi,\eta, \theta, \iota)}(k(\xi,\eta, \theta, \iota)[x,y,z])$.
We calculate its Gr\"obner bases.
The result is shown below. (We used the lexicographical
monomial order with $y>z>x$ for type $\mathbf{D_N^rB_-}$,
with $x>z>y$ for type $\mathbf{D_N^rC_-}$ 
and with $x>y>z$ for type $\mathbf{E_N^rF_4}$.
For type $\mathbf{D_N^r}$, we present here up to  
$\mathbf{D_{11}^0B_4}$ and $\mathbf{D_{11}^0C_9}$.) 

{\tiny

\vspace{2mm}

\noindent
$\mathbf{D_4^1B_1}:$ $y^2+zx+\eta x^2, yz+\xi x^2, yx+\theta x^2, z^2+\iota x^2, zx^2, x^3$.

\noindent
$\mathbf{D_4^1B_2}:$ $y^2+yz+\xi x^2, yx+\theta x^2, z^2+\iota x^2, zx+\eta x^2, x^3$.

\noindent
$\mathbf{D_4^0B_1}:$ $y^2,  z^2, x^2$.

\noindent
$\mathbf{D_4^0B_2}:$ $y^2,  z^2, x^2$.

\noindent
$\mathbf{D_6^2B_2}:$ $y^3+yz+\xi x^2, y^2z+zx+\eta x^2, yx+\theta x^2, z^2+\iota x^2, zx^2, x^3$.

\noindent
$\mathbf{D_6^2B_3}:$ $y^3+yz+\xi x^2, yx+\theta x^2, z^2+\iota x^2, zx+\eta x^2, x^3$.

\noindent
$\mathbf{D_6^1B_2}:$ $y^3+y^2z, y^2x, z^2,  x^2$.

\noindent
$\mathbf{D_6^1B_3}:$ $y^3+y^2z, y^2x, z^2,  x^2$.

\noindent
$\mathbf{D_7^2B_2}:$ $y^4+zx+\eta x^2, yz+\xi x^2, yx+\theta x^2, z^2+\iota x^2, zx^2, x^3$.

\noindent
$\mathbf{D_8^3B_3}:$ $y^4+yz+\xi x^2, y^3z+zx+\eta x^2, yx+\theta x^2, z^2+\iota x^2, zx^2, x^3$.

\noindent
$\mathbf{D_8^3B_4}:$ $y^4+yz+\xi x^2, yx+\theta x^2, z^2+\iota x^2, zx+\eta x^2,  x^3$.

\noindent
$\mathbf{D_8^2B_2}:$ $y^4+\eta x^2, y^2z+(\xi+\eta)x^2, y^2x+\theta x^2, z^2+\iota x^2, zx^2, x^3$.

\noindent
$\mathbf{D_8^2B_3}:$ $y^4+y^2z, y^2x, z^2, x^2$.

\noindent
$\mathbf{D_8^2B_4}:$ $y^4+y^2z, y^2x, z^2, x^2$.

\noindent
$\mathbf{D_8^1B_2}:$ $y^4+y^2zx+\eta x^2, y^3z+y^2zx+(\eta+\xi)x^2, y^3x+\theta x^2, yx^2, z^2+\iota x^2, zx^2, x^3$.

\noindent
$\mathbf{D_8^1B_3}:$ $y^4+y^3z+\xi x^2, y^3x+\theta x^2, y^2zx+\eta x^2, yx^2, z^2+\iota x^2, zx^2, x^3$.

\noindent
$\mathbf{D_8^1B_4}:$ $y^4+y^3z+\xi x^2, y^3x+\theta x^2, y^2zx+\eta x^2, yx^2, z^2+\iota x^2, zx^2, x^3$.

\noindent
$\mathbf{D_8^0B_2}:$ $y^4, z^2, x^2$.

\noindent
$\mathbf{D_8^0B_3}:$ $y^4, z^2, x^2$.

\noindent
$\mathbf{D_8^0B_4}:$ $y^4, z^2, x^2$.

\noindent
$\mathbf{D_9^3B_3}:$ $y^6+zx+\eta x^2, yz+\xi x^2, yx+\theta x^2, z^2+\iota x^2, zx^2, x^3$.

\noindent
$\mathbf{D_{10}^4B_4}:$ $y^5+yz+\xi x^2, y^4z+zx+\eta x^2, yx+\theta x^2, z^2+\iota x^2, zx^2, x^3$.

\noindent
$\mathbf{D_{10}^4B_5}:$ $y^5+yz+\xi x^2, yx+\theta x^2, z^2+\iota x^2, zx+\eta x^2, x^3$.

\noindent
$\mathbf{D_{10}^3B_3}:$ $y^5+y^2z+\xi x^2, y^3z+\xi yx^2+\eta x^2, y^2x+\theta x^2, z^2+\iota x^2, zx^2, x^3$.

\noindent
$\mathbf{D_{10}^3B_4}:$ $y^5+y^2z, y^2x, z^2, x^2$.

\noindent
$\mathbf{D_{10}^3B_5}:$ $y^5+y^2z, y^2x, z^2, x^2$.

\noindent
$\mathbf{D_{10}^2B_3}:$ $y^5+y^3z+\xi x^2, y^4z+y^2zx+\eta x^2, y^3x+\theta x^2, yx^2, z^2+\iota x^2, zx^2, x^3$.

\noindent
$\mathbf{D_{10}^2B_4}:$ $y^5+y^3z+\xi x^2, y^3x+\theta x^2, y^2zx+\eta x^2, yx^2, z^2+\iota x^2, zx^2, x^3$.

\noindent
$\mathbf{D_{10}^2B_5}:$ $y^5+y^3z+\xi x^2, y^3x+\theta x^2, y^2zx+\eta x^2, yx^2, z^2+\iota x^2, zx^2, x^3$.

\noindent
$\mathbf{D_{10}^1B_3}:$ $y^5+y^4z, y^4x, z^2, x^2$.

\noindent
$\mathbf{D_{10}^1B_4}:$ $y^5+y^4z, y^4x, z^2, x^2$.

\noindent
$\mathbf{D_{10}^1B_5}:$ $y^5+y^4z, y^4x, z^2, x^2$.

\noindent
$\mathbf{D_{11}^4B_4}:$ $y^8+zx+\eta x^2, yz+\xi x^2, yx+\theta x^2, z^2+\iota x^2, zx^2, x^3$.

$\quad\vdots$
\vspace{2mm}

\noindent
$\mathbf{D_5^1C_3}:$ $x^2+xz+\iota y^2, xy+y^2, z^2, zy, y^3$.

\noindent
$\mathbf{D_5^0C_3}:$ $x^2, z^2, y^2$.

\noindent
$\mathbf{D_6^2C_4}:$ $x^2+xz+\eta y^4, xy+\theta y^4, z^2+\iota y^4, zy+\xi y^4+y^3, y^5$.

\noindent
$\mathbf{D_6^1C_3}:$ $x^2+\iota y^2, xy^2+\theta y^3, z^2, zy^2+(1+\xi)y^3, y^4$.

\noindent
$\mathbf{D_6^1C_4}:$ $x^2, xy^2, z^2, zy^2+y^3, y^4$.

\noindent
$\mathbf{D_6^0C_3}:$ $x^2+(\xi +1) xy^2+\theta zy^2+\iota y^2, z^2, y^3$.

\noindent
$\mathbf{D_6^0C_4}:$ $x^2+xy^2, z^2, y^3$.

\noindent
$\mathbf{D_7^2C_5}:$ $x^2+xz+\iota y^4, xy+y^3, z^2, zy, y^5$. 

\noindent
$\mathbf{D_7^1C_4}:$ $x^2, xy^2+y^3, z^2, zy^2, y^4$.

\noindent
$\mathbf{D_7^1C_5}:$ $x^2, xy^2+y^3, z^2, zy^2, y^4$.

\noindent
$\mathbf{D_7^0C_4}:$ $x^2+zy^2, z^2, y^3$.

\noindent
$\mathbf{D_7^0C_5}:$ $x^2+zy^2, z^2, y^3$.

\noindent
$\mathbf{D_8^3C_6}:$ $x^2+xz+\eta y^6, xy+\theta y^6, z^2+\iota y^6, zy+\xi y^6+y^4, y^7$.

\noindent
$\mathbf{D_8^2C_4}:$ $x^2+\eta y^4, xy^2+\theta y^4, z^2+\iota y^4, zy^2+(1+\xi)y^4, y^6$.

\noindent
$\mathbf{D_8^2C_5}:$ $x^2+\iota y^4, xy^2+\theta y^5, z^2, zy^2+\xi y^5+y^4, y^6$.

\noindent
$\mathbf{D_8^2C_6}:$ $x^2, xy^2, z^2, zy^2+y^4, y^6$.

\noindent
$\mathbf{D_8^1C_4}:$ $x^2+xzy^2+\eta y^4, xy^3+\theta y^4, z^2+\iota y^4, zy^3+(1+\xi )y^4, y^5$.

\noindent
$\mathbf{D_8^1C_5}:$ $x^2+xzy^2+\iota y^4, xy^3, z^2, zy^3+y^4, y^5$.

\noindent
$\mathbf{D_8^1C_6}:$ $x^2+xzy^2, xy^3, z^2, zy^3+y^4, y^5$.

\noindent
$\mathbf{D_8^0C_4}:$ $x^2, z^2, y^4$.

\noindent
$\mathbf{D_8^0C_5}:$ $x^2, z^2, y^4$.

\noindent
$\mathbf{D_8^0C_6}:$ $x^2, z^2, y^4$.

\noindent
$\mathbf{D_9^3C_7}:$ $x^2+xz+\iota y^6, xy+y^4, z^2, zy, y^7$.

\noindent
$\mathbf{D_9^2C_5}:$ $x^2+\iota y^4, xy^2+\theta y^5+y^4, z^2, zy^2+\xi y^5, y^6$.

\noindent
$\mathbf{D_9^2C_6}:$ $x^2, xy^2+y^4, z^2, zy^2, y^6$.

\noindent
$\mathbf{D_9^2C_7}:$ $x^2, xy^2+y^4, z^2, zy^2, y^6$.

\noindent
$\mathbf{D_9^1C_5}:$ $x^2+xzy^2+\iota y^4, xy^3+y^4, z^2, zy^3, y^5$.

\noindent
$\mathbf{D_9^1C_6}:$ $x^2+xzy^2, xy^3+y^4, z^2, zy^3, y^5$.

\noindent
$\mathbf{D_9^1C_7}:$ $x^2+xzy^2, xy^3+y^4, z^2, zy^3, y^5$.

\noindent
$\mathbf{D_9^0C_5}:$ $x^2, z^2, y^4$.

\noindent
$\mathbf{D_9^0C_6}:$ $x^2, z^2, y^4$.

\noindent
$\mathbf{D_9^0C_7}:$ $x^2, z^2, y^4$.

\noindent
$\mathbf{D_{10}^4C_8}:$ $x^2+xz+\eta y^8, xy+\theta y^8, z^2+\iota y^8, zy+\xi y^8+y^5, y^9$.

\noindent
$\mathbf{D_{10}^3C_6}:$ $x^2+\eta y^6, xy^2+\theta y^6, z^2+\iota y^6, zy^2+\xi y^6+y^5, y^8$.

\noindent
$\mathbf{D_{10}^3C_7}:$ $x^2+\iota y^6, xy^2+\theta y^7, z^2, zy^2+\xi y^7+y^5, y^8$.

\noindent
$\mathbf{D_{10}^3C_8}:$ $x^2, xy^2, z^2, zy^2+y^5, y^8$.

\noindent
$\mathbf{D_{10}^2C_5}:$ $x^2+xzy^2+\iota y^4, xy^3+\theta y^5, z^2+\eta y^6, zy^3+(1+\xi)y^5, y^7$.

\noindent
$\mathbf{D_{10}^2C_6}:$ $x^2+xzy^2+\eta y^6, xy^3+\theta y^6, z^2+\iota y^6, zy^3+\xi y^6+y^5, y^7$.

\noindent
$\mathbf{D_{10}^2C_7}:$ $x^2+xzy^2+\iota y^6, xy^3, z^2, zy^3+y^5, y^7$.

\noindent
$\mathbf{D_{10}^2C_8}:$ $x^2+xzy^2, xy^3, z^2, zy^3+y^5, y^7$.

\noindent
$\mathbf{D_{10}^1C_5}:$ $x^2+\iota y^4, xy^4+\theta y^5, z^2, zy^4+(1+\xi) y^5, y^6$.

\noindent
$\mathbf{D_{10}^1C_6}:$ $x^2, xy^4, z^2, zy^4+y^5, y^6$.

\noindent
$\mathbf{D_{10}^1C_7}:$ $x^2, xy^4, z^2, zy^4+y^5, y^6$.

\noindent
$\mathbf{D_{10}^1C_8}:$ $x^2, xy^4, z^2, zy^4+y^5, y^6$.

\noindent
$\mathbf{D_{10}^0C_5}:$ $x^2+(\xi+1) xy^4+\theta zy^4+\iota y^4, z^2, y^5$.

\noindent
$\mathbf{D_{10}^0C_6}:$ $x^2+ xy^4, z^2, y^5$.

\noindent
$\mathbf{D_{10}^0C_7}:$ $x^2+ xy^4, z^2, y^5$.

\noindent
$\mathbf{D_{10}^0C_8}:$ $x^2+ xy^4, z^2, y^5$.

\noindent
$\mathbf{D_{11}^4C_9}:$ $x^2+xz+\iota y^8, xy+y^5, z^2, zy, y^9$.

\noindent
$\mathbf{D_{11}^3C_7}:$ $x^2+\iota y^6, xy^2+\theta y^7+y^5, z^2, zy^2+\xi y^7, y^8$.

\noindent
$\mathbf{D_{11}^3C_8}:$ $x^2, xy^2+y^5, z^2, zy^2, y^8$.

\noindent
$\mathbf{D_{11}^3C_9}:$ $x^2, xy^2+y^5, z^2, zy^2, y^8$.

\noindent
$\mathbf{D_{11}^2C_6}:$ $x^2+xzy^2+\eta y^6, xy^3+\theta y^6+y^5, z^2+\iota y^6, zy^3+\xi y^6, y^7$.

\noindent
$\mathbf{D_{11}^2C_7}:$ $x^2+xzy^2+\iota y^6, xy^3+y^5, z^2, zy^3, y^7$.

\noindent
$\mathbf{D_{11}^2C_8}:$ $x^2+xzy^2, xy^3+y^5, z^2, zy^3, y^7$.

\noindent
$\mathbf{D_{11}^2C_9}:$ $x^2+xzy^2, xy^3+y^5, z^2, zy^3, y^7$.

\noindent
$\mathbf{D_{11}^1C_6}:$ $x^2, xy^4+y^5, z^2, zy^4, y^6$.

\noindent
$\mathbf{D_{11}^1C_7}:$ $x^2, xy^4+y^5, z^2, zy^4, y^6$.

\noindent
$\mathbf{D_{11}^1C_8}:$ $x^2, xy^4+y^5, z^2, zy^4, y^6$.

\noindent
$\mathbf{D_{11}^1C_9}:$ $x^2, xy^4+y^5, z^2, zy^4, y^6$.

\noindent
$\mathbf{D_{11}^0C_6}:$ $x^2+zy^4, z^2, y^5$.

\noindent
$\mathbf{D_{11}^0C_7}:$ $x^2+zy^4, z^2, y^5$.

\noindent
$\mathbf{D_{11}^0C_8}:$ $x^2+zy^4, z^2, y^5$.

\noindent
$\mathbf{D_{11}^0C_9}:$ $x^2+zy^4, z^2, y^5$.

$\quad\vdots$

\noindent
$\mathbf{E_7^3F_4}:$ $x^2+y^3+yz+\xi z^2/\iota, xy+\theta z^2/\iota, 
xz+\eta z^2/\iota, yz+z^2/\iota,y^4+z^2/\iota, yz^2, z^3$.

\noindent
$\mathbf{E_7^2F_4}:$ $x^2, xy^2+y^2z, y^3, z^2$.

\noindent
$\mathbf{E_7^1F_4}:$ $x^2+y^3, xy^2, y^4, z^2$.

\noindent
$\mathbf{E_8^3F_4}:$ $x^2, xy^2, y^4, z^2$.

\noindent
$\mathbf{E_8^2F_4}:$ $x^2+y^2z, xy^2, y^4, z^2$.

\noindent
$\mathbf{E_8^1F_4}:$ $x^2+y^3z+\xi z^2/\iota, xy^3+\theta z^2/\iota, 
xy^2z+\eta z^2/\iota, xz^2, y^4+z^2/\iota, yz^2, z^3$.

\noindent
$\mathbf{E_8^0F_4}:$ $x^2, y^4, z^2$.

}

\medskip

\noindent
The Tjurina numbers can be read off from these Gr\"obner bases.
It is also checked that the associated graded ring
of the Tjurina algebra over $k(\xi,\eta, \theta, \iota)$
with respect to the nilradical $(x,y,z)$ is essentially 
obtained by coefficient extension of that of the closed
fiber, i.e.  one has the isomorphism   
\begin{eqnarray*}
&&\mathrm{gr\, }_{(x,y,z)} k(\xi,\eta, \theta, \iota)[x,y,z]/(\partial F/\partial x,\partial F/\partial y, \partial F/\partial z, F) \\
&&\quad\cong 
k(\xi,\eta, \theta, \iota)\otimes_k \mathrm{gr\, }_{(x,y,z)} k[x,y,z]/(\partial g/\partial x,\partial g/\partial y, \partial g/\partial z, g),
\end{eqnarray*}
where $g(x,y,z)=0$ is the defining equation of 
the rational double point on the closed fiber.
\end{remarks}

\begin{table}[h]
\caption{Non-classical compound Du Val singularities in $p=2$, I. (conti. $n\geq 4$) }
\label{tab:non-classical01conti}
\begin{center}
{\scriptsize 
\begin{tabular}{l|l|l|l}
\hline
   & \textrm{Type} & \textrm{Defining equation}& \textrm{Condition} \\
\hline
 $D_{2n}^{r}$ & $\mathbf{D_{2n}^{r}B_{r}}$  & $z^2+xy^n+xy^{n-r}z+wx^2+y^{2r+1}=0$  & $_{n/2 < r \leq n-1}$ \\
  & $\mathbf{D_{2n}^{r}B_{r+1}}$  & $z^2+xy^n+xy^{n-r}z+wx^2+y^{2r+3}=0$  &     \\
            & $\quad\vdots$  & $\quad\vdots$  &  \\
 & $\mathbf{D_{2n}^{r}B_{n-1}}$  & $z^2+xy^n+xy^{n-r}z+wx^2+y^{2n-1}=0$  &  \\
  &  $\mathbf{D_{2n}^{r}B_{n}}$  &  $z^2+xy^n+xy^{n-r}z+wx^2=0$ &  \\
 &  $\mathbf{D_{2n}^{r}C_{2r}}$  & $z^2+x^2y+xy^n+xy^{n-r}z+wy^{2r}=0$  &  \\
 &  $\mathbf{D_{2n}^{r}C_{2r+1}}$  & $z^2+x^2y+xy^n+xy^{n-r}z+wy^{2r+1}=0$  &  \\
            & $\quad\vdots$  & $\quad\vdots$  &  \\
 &  $\mathbf{D_{2n}^{r}C_{2n-2}}$  & $z^2+x^2y+xy^n+xy^{n-r}z+wy^{2n-2}=0$  &  \\
\cline{2-3}
 $\quad\vdots$ & $\quad\vdots$  & $\quad\vdots$  &  \\
\cline{2-3}
 $D_{2n}^{r}$ & $\mathbf{D_{2n}^{r}B_{\lfloor (n+1)/2\rfloor}}$  & $z^2+xy^n+xy^{n-r}z+wx^2+y^{2\lfloor(n+1)/2\rfloor+1}=0$  &  $_{1\leq r\leq n/2}$ \\
 & $\mathbf{D_{2n}^{r}B_{\lfloor (n+1)/2\rfloor+1}}$  & $z^2+xy^n+xy^{n-r}z+wx^2+y^{2\lfloor (n+1)/2\rfloor+3}=0$  &  \\
            & $\quad\vdots$  & $\quad\vdots$  &  $_{\lfloor\phantom{11}\rfloor\textrm{ stands for}}$ \\
  & $\mathbf{D_{2n}^{r}B_{n-1}}$  & $z^2+xy^n+xy^{n-r}z+wx^2+y^{2n-1}=0$  &  $^{\textrm{rounding down.}}$  \\
  &  $\mathbf{D_{2n}^{r}B_{n}}$  &  $z^2+xy^n+xy^{n-r}z+wx^2=0$ &  \\
 &  $\mathbf{D_{2n}^{r}C_{n}}$  &  $z^2+x^2y+xy^n+xy^{n-r}z+wy^{n}=0$  &  \\
 &  $\mathbf{D_{2n}^{r}C_{n+1}}$  & $z^2+x^2y+xy^n+xy^{n-r}z+wy^{n+1}=0$  &  \\
            & $\quad\vdots$  & $\quad\vdots$  &  \\
 &  $\mathbf{D_{2n}^{r}C_{2n-3}}$  & $z^2+x^2y+xy^n+xy^{n-r}z+wy^{2n-3}=0$  &  \\
 &  $\mathbf{D_{2n}^{r}C_{2n-2}}$  & $z^2+x^2y+xy^n+xy^{n-r}z+wy^{2n-2}=0$  &  \\
\cline{2-3}
 $\quad\vdots$ & $\quad\vdots$  & $\quad\vdots$  &  \\
\cline{2-3}
 $D_{2n}^0$ & $\mathbf{D_{2n}^0B_{n/2}}$  &  $z^2+xy^n+wx^2+y^{n+1}=0$ & $_{\textrm{$n$ is even.}}$  \\
 & $\mathbf{D_{2n}^0B_{(n+2)/2}}$  &  $z^2+xy^n+wx^2+y^{n+3}=0$ & $_{ \textrm{$n$ is even.}}$  \\
             & $\quad\vdots$  & $\quad\vdots$  & $\quad\vdots$ \\
           & $\mathbf{D_{2n}^0B_{n-1}}$  &  $z^2+xy^n+wx^2+y^{2n-1}=0$ & $_{ \textrm{$n$ is even.}}$ \\
           & $\mathbf{D_{2n}^0B_{n}}$  &  $z^2+xy^n+wx^2=0$ & $_{ \textrm{$n$ is even.}}$ \\
           & $\mathbf{D_{2n}^0C_{n}}$  &  $z^2+x^2y+xy^n+wy^{n}=0$ &  \\
           & $\mathbf{D_{2n}^0C_{n+1}}$  &  $z^2+x^2y+xy^n+wy^{n+1}=0$ &  \\
            & $\quad\vdots$  & $\quad\vdots$  &  \\
           & $\mathbf{D_{2n}^0C_{2n-3}}$  &  $z^2+x^2y+xy^n+wy^{2n-3}=0$ &  \\
           & $\mathbf{D_{2n}^0C_{2n-2}}$  &  $z^2+x^2y+xy^n+wy^{2n-2}=0$ &  \\
\hline
 $D_{2n+1}^{n-1}$ & $\mathbf{D_{2n+1}^{n-1}B_{n-1}}$  & {\tiny $z^2+xyz+wx^2+y^{2n-1}=0$} &  \\
  & $\mathbf{D_{2n+1}^{n-1}C_{2n-1}}$  & {\tiny $z^2+x^2y+y^nz+xyz+wy^{2n-1}=0$} &  \\
\cline{2-3}
 $\quad\vdots$ & $\quad\vdots$  & $\quad\vdots$  &  \\
\cline{2-3}
 $D_{2n+1}^r$ & $\mathbf{D_{2n+1}^rC_{2r+1}}$ & $z^2+x^2y+xy^{n-r}z+y^nz+wy^{2r+1}=0$ &  $_{n/2 <  r < n-1}$\\
   & $\mathbf{D_{2n+1}^rC_{2r+2}}$ & $z^2+x^2y+xy^{n-r}z+y^nz+wy^{2r+2}=0$ &    \\
             & $\quad\vdots$  & $\quad\vdots$  &  \\
            & $\mathbf{D_{2n+1}^rC_{2n-2}}$ & $z^2+x^2y+xy^{n-r}z+y^nz+wy^{2n-2}=0$ &  \\
            & $\mathbf{D_{2n+1}^rC_{2n-1}}$ & $z^2+x^2y+xy^{n-r}z+y^nz+wy^{2n-1}=0$ &  \\
\cline{2-3}
 $\quad\vdots$ & $\quad\vdots$  & $\quad\vdots$  &  \\
\cline{2-3}
 $D_{2n+1}^r$ & $\mathbf{D_{2n+1}^rC_{n+1}}$ & $z^2+x^2y+xy^{n-r}z+y^nz+wy^{n+1}=0$ & $_{1\leq r \leq n/2}$  \\
           & $\mathbf{D_{2n+1}^rC_{n+2}}$ & $z^2+x^2y+xy^{n-r}z+y^nz+wy^{n+2}=0$ &  \\
             & $\quad\vdots$  & $\quad\vdots$  &  \\
            & $\mathbf{D_{2n+1}^rC_{2n-2}}$ & $z^2+x^2y+xy^{n-r}z+y^nz+wy^{2n-2}=0$ &  \\
            & $\mathbf{D_{2n+1}^rC_{2n-1}}$ & $z^2+x^2y+xy^{n-r}z+y^nz+wy^{2n-1}=0$ &  \\
\cline{2-3}
 $\quad\vdots$ & $\quad\vdots$  & $\quad\vdots$  &  \\
\cline{2-3}
 $D_{2n+1}^0$ & $\mathbf{D_{2n+1}^0C_{n+1}}$ & $z^2+x^2y+y^nz+wy^{n+1}=0$ &  \\
            & $\mathbf{D_{2n+1}^0C_{n+2}}$ & $z^2+x^2y+y^nz+wy^{n+2}=0$ &  \\
             & $\quad\vdots$  & $\quad\vdots$  &  \\
            & $\mathbf{D_{2n+1}^0C_{2n-2}}$ & $z^2+x^2y+y^nz+wy^{2n-2}=0$ &  \\
            & $\mathbf{D_{2n+1}^0C_{2n-1}}$ & $z^2+x^2y+y^nz+wy^{2n-1}=0$ &  \\
\hline
\end{tabular}
}
\end{center}
\end{table}

\section{On rationality}

Whether Grauert-Riemenschneider vanishing theorem holds for our examples
needs to be clarified.

\begin{corollary}\label{vanishing}
Let $X$ be a hypersurface singularity given by one of the equations 
in Theorem~\ref{main}.
Then the following assertions hold.
\begin{itemize}
\item [i)] $R^i\pi_*\mathcal O_{\tilde X}=0$ $(i>0)$ holds for any resolution of singularities
$\pi:\tilde X\to X$. 
\item [ii)] $R^i\pi_*K_{\tilde X}=0$ $(i>0)$ holds for any resolution of singularities
$\pi:\tilde X\to X$. 
\end{itemize}
\end{corollary}

\begin{proof}
First note that assertion i) and ii) are 
equivalent to each other because
we have a crepant resolution $\pi:\tilde X\to X$, 
from which follows the equality $K_{\tilde X}\cong \mathcal O_{\tilde X}$. 

From the previous theorem we know that a general hyperplane section 
$H\subset X$
has a rational double point.
Let $\pi_1: X_1\to X$ be the blow-up along the singular locus of $X$.
Bertini's theorem tells us that this $\pi_1$ restricted to $H$ 
is also a point blow-up.    
The locus of 
a rational double point of $\pi_1^*H$
corresponds to the singular locus of $X_1$.
Then a general hyperplane section $H_1\subset X_1$ also has 
a rational double point, and direct calculation of cohomologies 
gives the vanishing.
\end{proof}

It might be worth mentioning that our examples are 
irrational from the  viewpoint of the theory of
tight closures. 

\begin{proposition}\label{f-rationality}
Let $X$ be a hypersurface singularity given by one of the equations 
in Theorem~\ref{main}.
Then the following assertions hold.
\begin{itemize}
\item [i)] $X$ is $F$-pure if and only if 
the type is one of 
{\footnotesize
$\mathbf{E_7^3F_4}$, $\mathbf{D_4^1B_1}$, 
$\mathbf{D_4^1B_2}$, $\mathbf{D_5^1C_3}$, $\mathbf{D_{2n}^{n-1}B_{n-1}}$,
$\mathbf{D_{2n}^{n-1}B_{n}}$, $\mathbf{D_{2n}^{n-1}C_{2n-2}}$,
$\mathbf{D_{2n+1}^{n-1}B_{n-1}}$,   
$\mathbf{D_{2n+1}^{n-1}C_{2n-1}}$ 
}
with $n\geq 3$.
\item [ii)] $X$ is not $F$-rational.
\end{itemize}
\end{proposition}

\begin{proof}
i) One uses Fedder's criterion for $F$-purity \cite[Proposition 2.1]{Fedder83}.

\noindent
ii) One needs to check that the singularities which turn out to be $F$-pure in i)
do not satisfy the criterion for $F$-regularity
\cite[Theorem 4.1.1]{Glassbrenner92}, \cite[Proposition 3.1]{Hara95}.
Then use the equivalence of $F$-regularity and $F$-rationality
for Gorenstein rings \cite[\S 10, Proposition 10.3.7]{BrunsHerzog98}.
\end{proof}

\begin{remark}
We presented here a direct proof on $F$-purity and $F$-rationality.
But this result can be derived from the knowledge of 
types of rational double 
points on general hyperplane sections only.
For example, combine \cite[Theorem 3.4]{Fedder83}, 
\cite[Proposition 10.3.11]{BrunsHerzog98}
and \cite[Theorem 6.1]{SchwedeZhang13}.
Then as Hara points out \cite[Remark~1.3]{Hara98}, 
one can determine the $F$-purity and $F$-regularity of 
rational double points using the criteria for 
the defining equations given by Artin \cite{Artin77}.
 
\end{remark}

\specialsection*{Acknowledgement}
The author would like to express his sincere gratitude to 
Professors Kei-ichi~Watanabe, Shihoko~Ishii,
Tadashi~Tomaru, Masataka~Tomari for valuable suggestions and comments.
I also thank Professors Kenji~Matsuki, Noboru~Nakayama, 
Takehiko~Yasuda, Shunsuke~Takagi,
Nobuo~Hara for discussion and comments and
Professors Natsuo~Saito, Hiroyuki~Ito and Toshiyuki~Katsura for their genuine support.

\bibliographystyle{amsplain}

\begin{thebibliography}{10}




\bibitem{Artin66} M.~Artin,
\textit{On isolated rational singularities of surfaces},
Amer. J. Math. \textbf{88} (1966), 129--136. 

\bibitem{Artin77} M.~Artin,
\textit{Coverings of the rational double points in characteristic $p$},
in: W.~L.~Baily Jr. and T.~Shioda (eds.), 
Complex Analysis and Algebraic Geometry,
Cambridge Univ. Press, Cambridge, 1977, pp. 11--22.

\bibitem{BrunsHerzog98} W. Bruns and H. Herzog,
Cohen-Macaulay Rings, 
Cambridge Univ. Press, Cambridge, 1998.

\bibitem{CossartPiltant08} V.~Cossart and O.~Piltant, 
\textit{Resolution of 
singularities of threefolds in positive characteristic.\,~{\rm I}},
J. Algebra, \textbf{320} (2008), 1051--1082.

\bibitem{CossartPiltant09} V.~Cossart and O.~Piltant, 
\textit{Resolution of 
singularities of threefolds in positive characteristic.\,~{\rm II}},
J. Algebra, \textbf{321} (2009), 1836--1976.

\bibitem{Cutkosky09} S. D. Cutkosky, 
\textit{Resolution of singularities for $3$-folds in positive characteristic}, 
Amer. J. Math. \textbf{131} (2009), 59--127.

\bibitem{Fedder83} R. Fedder, 
\textit{$F$-purity and rational singularity}, 
Trans. Amer. Math. Soc. \textbf{278} (1983), 461--480.

\bibitem{Glassbrenner92} D. Glassbrenner, 
\textit{Invariant rings of group actions, determinant rings, and tight closure},
Ph.~D. thesis, University of Michigan, 1992.


\bibitem{Hara95} N. Hara, 
\textit{$F$-regularity and $F$-purity of graded rings}, 
J. Algebra \textbf{172} (1995), 804--818.

\bibitem{Hara98} N. Hara, 
\textit{Classification of two-dimensional $F$-regular and 
$F$-pure singularities}, 
Adv. Math. \textbf{133} (1998), 33--53.

\bibitem{Hartshorne77} R.~Hartshorne, 
Algebraic Geometry,
Springer-Verlag, New York-Heidelberg, 1977.

\bibitem{HIS13} M. Hirokado, H. Ito, N. Saito, 
\textit{Three dimensional canonical singularities in codimension two in positive characteristic},
J. Algebra \textbf{373} (2013), 207--222. 

\bibitem{Lipman69} J. Lipman, 
\textit{Rational singularities with applications to algebraic surfaces and unique factorization}, 
Publ. Math. Inst. Hautes \'Etudes Sci. \textbf{9} (1969), 195--270.

\bibitem{Matsumura86} H. Matsumura,
Commutative Ring Theory, 
Cambridge Univ. Press, Cambridge, 1986.

\bibitem{Reid80} M.~Reid, 
\textit{Canonical $3$-folds}, 
in: A.~Beauville (ed.), 
Journ\'ees de g\'eom\'etrie\  alg\'ebrique d'Angers 1979,  
Sijthoff \& Noordhoff, Alphen, 1980, pp. 273--310.

\bibitem{SchwedeZhang13} K. Schwede and W. Zhang, 
\textit{Bertini theorems for $F$-singularities}, 
Proc. London Math. Soc. (3) \textbf{107} (2013), 851--874.

\end{thebibliography}

\end{document}